\documentclass[12pt]{amsart}
\usepackage[utf8]{inputenc}
\usepackage[T1]{fontenc}
\usepackage{amsmath,amssymb,amsthm,mathrsfs}
\usepackage{amsfonts}

\usepackage{here}
\usepackage{physics}
\usepackage{tikz}
\usepackage{mathdots}
\usepackage{yhmath}
\usepackage{cancel}
\usepackage{color}
\usepackage{siunitx}
\usepackage{array}
\usepackage{multirow}
\usepackage{gensymb}
\usepackage{tabularx}
\usepackage{extarrows}
\usepackage{booktabs}
\usetikzlibrary{fadings}
\usetikzlibrary{patterns}
\usetikzlibrary{shadows.blur}
\usetikzlibrary{shapes}

\newcommand{\Hom}{\mathrm{Hom}}
\newcommand{\Ext}{\mathrm{Ext}}
\newcommand{\Ker}{\mathrm{Ker}}

\newcommand{\coker}{\mathrm{Coker}}

\newcommand{\leib}{\mathrm{Leib}}
\newcommand{\lie}{\mathrm{Lie}}

\newcommand{\CL}{\mathrm{CL}}

\newcommand{\HL}{\mathrm{HL}}

\newcommand{\K}{\mathbb{K}}
\newcommand{\F}{\mathbb{F}}
\newcommand{\C}{\mathbb{C}}

\theoremstyle{remark}
\newtheorem{rem}{Remark}[subsection]

\theoremstyle{definition}
\newtheorem{defi}{Definition}[subsection]
\newtheorem{theo}{Theorem}[subsection]
\newtheorem{corol}{Corollary}[subsection]
\newtheorem{lem}{Lemma}[subsection]
\newtheorem{prop}{Proposition}[subsection]

\usepackage{geometry}
\usepackage{enumerate}

\begin{document}

\title{On the Gabriel quiver of extensions of Leibniz algebras}

\author{Ziwendtaor\'e Hermann Bamogo}

\author{Friedrich Wagemann}
\address{Laboratoire de math\'ematiques Jean Leray, UMR 6629 du CNRS, Universit\'e
de Nantes, 2, rue de la Houssini\`ere, F-44322 Nantes Cedex 3, France}
\email{wagemann@math.univ-nantes.fr}



          
\maketitle


\begin{abstract}
We compute the Gabriel quiver of simple objects in the category of bimodules over a simple Leibniz algebra and over the trivial $1$-dimensional Leibniz algebra. Vertices of the quiver are the classes of simple objects, arrows are given by the dimensions of $\Ext^1$-groups. 
\end{abstract}


\section{Introduction}

The notion of Leibniz algebra ${\mathfrak h}$ is a generalization of the notion of Lie algebra. The natural module theory for Leibniz algebras is a theory of bimodules. In this article, we are interested in the category of (finite dimensional) bimodules over a (finite dimensional) Leibniz algebra, more precisely in the Ext-groups in this category. Loday and Pirashvili performed the first computation for the Ext-groups $\Ext_{UL\left( \mathfrak{h} \right)}^{n}\left( M, N \right)$ between simple finite-dimensional Leibniz $\mathfrak{h}$-bimodules $M$ and $N$ over a simple Lie algebra $\mathfrak{g}=\mathfrak{h}$ (see \cite{JL-TP1}). Here $UL\left( \mathfrak{h} \right)$ is the enveloping Leibniz algebra (in the sense of Loday-Pirashvili, see Definition \ref{envelop}). The Leibniz ${\mathfrak h}$-bimodules then become simply the left $UL\left( \mathfrak{h} \right)$-modules in the usual sense.

Later, Mugni\'ery and Wagemann (see \cite{JM-FW2}) have computed the Ext-groups for a simple (non-Lie) Leibniz algebra $\mathfrak{h}$, following cohomology computations of Feldvoss-Wagemann (see \cite{JF-FW1}).
One feature of these computations is that they permit an explicit computation of the Gabriel quiver of these categories of bimodules. The {\it Gabriel quiver} is a directed graph whose vertices are the isomorphism classes of simple objects of the category of bimodules, and such that for two simple objects $M$ and $N$, the quiver has exactly $\dim\Ext_{UL\left( \mathfrak{h} \right)}^{1}\left( M, N \right)$ arrows from (the class of) $M$ to (the class of) $N$.

In this article, we compute the Ext-groups for the trivial $1$-dimensional complex Lie/Leibniz algebra $\mathfrak{h}=\mathbb{C}$, based on the cohomology computations of Feldvoss-Wagemann (see \cite{JF-FW2}). These computations are then used for the computation of the Gabriel quiver of this category of bimodules. 
The computations are, to our knowledge, the first computations of the Ext-groups for a non-simple Leibniz algebra. The Ext-dimension of the category of bimodules for the trivial $1$-dimensional complex Lie/Leibniz algebra is infinite. 

The computations of the Gabriel quiver is based on the following steps:
\begin{enumerate}
\item[$\bullet$] Compute the cohomology of $\mathfrak{h}$ with values in all symmetric bimodules $M^{s}$ and all antisymmetric bimodules $M^{a}$ for a simple left $\mathfrak{h}_{\lie}$-module $M$.
\item[$\bullet$] Use the change-of-ring spectral sequences (provided by Loday-Pirashvili in \cite{JL-TP1} and Mugni\'ery-Wagemann in \cite{JM-FW2}) to deduce the Ext-groups from the cohomology, showing that the spectral sequences collapse at the second page.
\item[$\bullet$] Use the computation of the Ext-groups to deduce the Gabriel quiver.
\end{enumerate}
In a second part of this article, we compute the Gabriel quiver for the hemi-semidirect product Leibniz algebra $\mathfrak{h}= V_{n} \times_{hs}\mathfrak{sl}_{2}(\C)$ for which the Ext-groups have been computed in \cite{JM-FW2}, and we complete the main theorem of \cite{JM-FW2}. 

The content of this article is as follows. We first recall in Section 2 some definitions and properties of Leibniz algebras and their bimodules. We then compute in Section 3 the Gabriel quiver for the trivial $1$-dimensional Leibniz algebra $\mathfrak{h}=\mathbb{C}$ (see Theorem \ref{comput}). After that, we compute in Section 4 the Gabriel quiver for the hemi-semidirect product $\mathfrak{h}= V_{n} \times_{hs}\mathfrak{sl}_{2}(\C)$. In the course of the computation, we give a complement to the main theorem of Mugniéry-Wagemann (see Theorem \ref{complement}). 

{\bf Acknowledgements:} This works constitutes Z. H. Bamogo's master thesis at Nantes Universit\'e which we thank for the financial support. This work was conducted within the France 2030 framework program, Centre Henri Lebesgue ANR-11-LABX-0020-01.

\section{Leibniz Algebras and their properties}

\subsection{Leibniz algebras}

Let $\K$ be a (commutative) field. We will suppose after the preliminaries that $\K=\C$, the field of complex numbers (especially for all structure results on simple Leibniz algebras and all results on representation theory). 

\begin{defi}
A left Leibniz algebra $\mathfrak{h}$ over a field $\K$ is a vector space endowed with a bilinear map
$$ \left[- ,- \right] : \mathfrak{h} \times \mathfrak{h} \longrightarrow \mathfrak{h}$$
such that $\forall$ $x,y,z \in \mathfrak{h}$ 
$$\left[ x , \left[ y, z \right] \right]= \left[ \left[ x, y \right],z  \right] + \left[y, \left[ x, z \right] \right].$$ 
\end{defi}

Similarly, one can define a right Leibniz algebra. In this case, Leibniz identity is the following:
$\forall$ $x,y,z \in {\mathfrak h}$
$$ [[x,y],z]=[[x,z],y] + [x,[y,z]].$$
Observe that Loday and Pirashvili always worked with right Leibniz algebras. 
In this work, we only consider left Leibniz algebras and we will just call them Leibniz algebras. In order to distinguish in notation Lie algebras from Leibniz algebras, we will denote the former by ${\mathfrak g}$, while we denote the latter by ${\mathfrak h}$. 

\begin{defi}
Every Leibniz algebra $\mathfrak{h}$ has a special ideal called the {\it Leibniz kernel} or the {\it ideal of squares} of $\mathfrak{h}$. We denote it by $\leib \left( \mathfrak{h} \right)$:
$$ \leib \left( \mathfrak{h} \right)= \langle \left[ x, x \right] \mid x \in \mathfrak{h} \rangle_{\K} $$
\end{defi}  

\begin{defi}
To every Leibniz algebra $\mathfrak{h} $ is associated its canonical Lie algebra. We denote it by $\mathfrak{h}_{\lie}$:
$$\mathfrak{h}_{\lie}= \mathfrak{h}/\leib \left( \mathfrak{h} \right) $$
\end{defi}

\begin{defi}
Let $\mathfrak{g}$ be a Lie algebra and $M$ a $\mathfrak{g}$-module. The direct sum $M \oplus \mathfrak{g}$ becomes a Leibniz algebra, called the {\it hemi-semidirect product} of ${\mathfrak g}$ and $M$, by endowing it with the following bracket operation.
For $x,y \in \mathfrak{g}$ and $a,b \in M$
$$\left[ \left(a,x \right), \left(b,y \right) \right]= \left(x.b, \left[x,y \right] \right)  $$

We denote the hemi-semidirect product of $\mathfrak{g}$ and $M$  by  $M \times_{hs}\mathfrak{g}$
\end{defi}

\begin{defi}
A Leibniz algebra $\mathfrak{h}$ is said to be \it{simple} if $0$, $\mathfrak{h}$ and $\leib \left( \mathfrak{h} \right)$ are the only two-sided ideals of $\mathfrak{h}$ with the condition $\leib \left( \mathfrak{h} \right) \subsetneqq \left[ \mathfrak{h}, \mathfrak{h} \right] $. 
\end{defi}

\begin{prop}[See Proposition 7.2, in \cite{JF}]
Let $\mathfrak{h}$ be a simple Leibniz algebra. Then $\mathfrak{h}_{\lie}$ is a simple Lie algebra and $\leib \left( \mathfrak{h} \right)$ is a simple $\mathfrak{h}_{\lie}$-module. Conversely, the hemi-semidirect product of a simple Lie algebra ${\mathfrak g}$ and a simple ${\mathfrak g}$-module $M$ is a simple Leibniz algebra.  
\end{prop}

\subsection{Leibniz bimodules}

The natural concept of a module for Leibniz algebras is a {\it bimodule}:

\begin{defi}
Let $\mathfrak{h}$ be a Leibniz algebra over a field $\K$. A left $\mathfrak{h}$-module is a vector space $M$ over $\K$ endowed with the following $\K$-bilinear left operation
$$\mathfrak{h} \times M \longrightarrow M \quad \left( x,m \right)\mapsto x\cdot m,  $$

satisfying 
$$\left[ x, y \right]\cdot m= x\cdot ( y\cdot m) - y\cdot (x\cdot m)$$ 
$\forall m \in M$ $\forall x,y \in \mathfrak{h}$.
\end{defi}

\begin{defi} \label{r and l }
An $\mathfrak{h}$-bimodule is a vector space $M$ with a left operation and a right operation of $\mathfrak{h}$ on $M$ satisfying the following conditions:

$$ (LLM) \quad \left[ x, y \right]\cdot m = x\cdot \left( y\cdot m \right)- y\cdot \left( x\cdot m \right)$$
$$ (LML) \quad \left( x\cdot m \right)\cdot y = x\cdot \left( m\cdot y \right)- m\cdot \left[ x, y \right]$$
$$ (MLL) \quad \left( m\cdot x \right)\cdot y = m\cdot \left[ x, y \right] - x\cdot \left( m\cdot y \right)$$

$\forall m \in M$ \quad $\forall x,y \in \mathfrak{h}$.
\end{defi}

 As a consequence of $(LLM)$, every Leibniz bimodule is a Leibniz module. In fact, if $M$ is an $\mathfrak{h}$-module, $M$ has a natural $\mathfrak{h}_{\lie}$-module structure and vice-versa (See Lemma 3.3, in \cite{JF} ).

\begin{rem}
Given an $\mathfrak{h}_{\lie}$-module $M$, there are two ways to induce a Leibniz $\mathfrak{h}$-bimodule:
Setting  $m.x=-x.m \quad \forall m \in M \quad \forall x \in \mathfrak{h} $,
one obtains a {\it symmetric bimodule} $M^{s}$.
Setting  $m.x=0 \quad \forall m \in M \quad \forall x \in \mathfrak{h} $,
one obtains an {\it antisymmetric bimodule} $M^{a}$.
\end{rem}

Let $M$ be an arbitrary ${\mathfrak h}$-bimodule. The {\it antisymmetric kernel} of $M$ is denoted $M_{0}$ and defined by 
$$M_{0}= \langle x.m + m.x \mid x \in \mathfrak{h} \quad m \in M  \rangle.$$
The subspace $M_{0}$ is an antisymmetric subbimodule and the quotient $M^{sym}:=M/M_{0}$ is symmetric. (see Proposition 3.12 and Proposition 3.13, in \cite{JF}).
This yields for any ${\mathfrak h}$-bimodule $M$  an exact sequence
$$ 0 \rightarrow M_{0} \rightarrow M \rightarrow M^{sym}=M/M_{0} \rightarrow 0.$$
As a consequence, we obtain the following structure result on simple ${\mathfrak h}$-bimodules: 

\begin{lem}[Loday-Pirashvili]
Let $\mathfrak{h}$ be a Leibniz algebra and $M$ an $\mathfrak{h}$-bimodule. Then\\
$M$ is simple $\Leftrightarrow$ $M$ is symmetric or antisymmetric for a simple $\mathfrak{h}_{\lie}$-module $M$.
\end{lem}

\subsection{Enveloping algebra}

Loday and Pirashvili have defined the universal enveloping algebra of a Leibniz algebra, denoted $UL\left( \mathfrak{h} \right)$, such that the left $UL\left( \mathfrak{h} \right)$-modules are exactly the Leibniz $\mathfrak{h}$-bimodules (see Theorem 2.3, in \cite{JL-TP2}).

\begin{defi} \label{envelop}(For the right setting, see \cite{JL-TP1}, for the left setting, see Definition 1.8, in \cite{JM-FW2})
Let $\mathfrak{h}$ be a Leibniz algebra. Given two copies $\mathfrak{h}^{l}$ and $\mathfrak{h}^{r}$ of $\mathfrak{h}$ generated respectively by the elements $l_{x}$ and $r_{x}$ for $x \in \mathfrak{h}$, the {\it universal enveloping algebra} is defined as the unital associative algebra
$$UL\left( \mathfrak{h} \right)= T\left( \mathfrak{h}^{l} \oplus \mathfrak{h}^{r}\right)/J $$
where 
$$T\left( \mathfrak{h}^{l} \oplus \mathfrak{h}^{r}\right)= \bigoplus_{n=0}^{\infty} \left( \mathfrak{h}^{l} \oplus \mathfrak{h}^{r}\right)^{\otimes n}$$
 is the tensor algebra of $\mathfrak{h}^{l} \oplus \mathfrak{h}^{r}$ and $J$ is the two-sided ideal of 
 $T\left( \mathfrak{h}^{l} \oplus \mathfrak{h}^{r}\right)$ generated by the elements:
 
 \begin{enumerate}
 \item[(i)] $l_{\left[ x, y \right]}- l_{x}\otimes l_{y} + l_{y}\otimes l_{x}$ ,
\item[(ii)]  $r_{\left[ x, y \right]}- l_{x}\otimes r_{y} + r_{y}\otimes l_{x}$, 
\item[(iii)] $r_{y}\otimes \left( l_{x} + r_{x}\right)$.
\end{enumerate}

\end{defi} 

\begin{rem}
Observe that $l_{x}$ and $r_{x}$ correspond respectively to the left operation and the right operation in Definition $\ref{r and l }$. Relation (iii) corresponds to the sum of the relations (LML) and (MLL). 
\end{rem}

\subsection{Leibniz cohomology}

We define the cohomology of Leibniz algebras here for the left setting (see \cite{JM-FW2}). For the right setting, see \cite{JL-TP2}.

Let $\mathfrak{h}$ be a Leibniz algebra and $M$ be an $\mathfrak{h}$-bimodule.
Denote $\CL^{n}\left( \mathfrak{h},M \right):=\Hom\left( \mathfrak{h}^{\otimes n},M \right)$ and 
$d^{n}: \CL^{n}\left( \mathfrak{h},M \right)\longrightarrow \CL^{n+1}\left( \mathfrak{h},M \right)$. 
The $\K$-linear map $d^{n}$ is defined as follows: 

\begin{align*}
d^{n}f\left( x_{0},\ldots,x_{n} \right)&=& \sum_{i=0}^{n-1}\left( -1 \right)^{i} x_{i}\cdot f\left( x_{0},\ldots,\widehat{x_{i}},\ldots,x_{n} \right) + \left( -1 \right)^{n-1}  f\left( x_{0},\ldots,x_{n-1}\right)\cdot x_{n} +\\
&+&\sum_{0\leq i < j \leq n}\left( -1 \right)^{i+1} f\left( x_{0},\ldots,\widehat{x_{i}},\ldots,x_{j-1},\left[ x_{i}, x_{j} \right],x_{j+1},\ldots, x_{n} \right),
\end{align*}

and we have  $d^{n+1} \circ d^{n}=0$. Therefore $\lbrace \CL^{n}\left( \mathfrak{h},M \right), d^{n} \rbrace_{n\geq0}$ is a cochain complex and its cohomology, denoted $\HL^n({\mathfrak h},M)$, is the cohomology of $\mathfrak{h}$ with coefficients in $M$.

By definition $\CL^{0}\left( \mathfrak{h},M \right)=M$ and $d^{0}m\left( x \right)= - m\cdot x $, thus
$\HL^{0}\left( \mathfrak{h},M \right)$ is the subspace of right invariants, also denoted $M^{\mathfrak h}$. For example, if $M$ is antisymmetric, $\HL^{0}\left( \mathfrak{h},M \right)=M$.


\section{Gabriel quiver for the trivial Lie/Leibniz algebra}

The construction of the Gabriel quiver is based on the computation of Ext-groups between simple bimodules. So, we will first collect some results in order to compute the Ext-groups for the trivial Lie/Leibniz algebra ${\mathfrak h}=\K$.
 
\subsection{Simple bimodules of the trivial Lie/Leibniz algebra}

\begin{defi}
Let $\mathfrak{h}=\K$ be the trivial Lie/Leibniz algebra. Let $e=1 \in \K$ such that $\K=\langle e \rangle$. Let $M$ be a $\K$-bimodule. The bimodule $M$ is a {\it trivial} $\K$-bimodule  if:
$$e.m=0 \quad {\rm and} \quad m.e=0 \quad \forall m \in M.$$
The bimodule $M$ is a {\it non-trivial} $K$-bimodule if:
$$\exists m \in M \quad e.m \neq 0 .$$
\end{defi}

It follows from the existence of eigenvectors that the simple $\K$-bimodules when $\K=\mathbb{C}$ are $1$-dimensional, that the simple $\K$-bimodules are $1$- or $2$-dimensional if $\K=\mathbb{R}$, and that there are simple $\K$-bimodules of any dimension for $\K=\mathbb{Q}$ and $\K=\F_{q}$ (the finite field of $q=p^n$ elements).

\begin{lem} \label{schur}[Schur's Lemma]
Let ${\mathfrak{g}}$ be a Lie algebra over $\mathbb{C}$ with universal enveloping algebra $U{\mathfrak{g}}$ (in the usual sense). Let $V$ and $W$ be two simple $U{\mathfrak{g}}$-modules. Then

$$\dim \Hom_{U{\mathfrak{g}}}\left( V, W \right) =
\begin{cases}
0 \quad {\rm if} \quad V \not \cong W  \\
1 \quad {\rm if} \quad V \cong W 
\end{cases}$$
\end{lem}

\begin{corol}\label{corol schur}
Let $\mathfrak{h}$ be a simple Leibniz algebra (over $\K=\C$). Let $N$ be a simple $\mathfrak{h}$-bimodule.

$$\Hom_{U{\mathfrak{h}_{\lie}}}\left( \mathfrak{h}, N \right) \cong
\begin{cases}
\K \quad {\rm if} \quad N \cong \mathfrak{h}_{\lie} \quad {\rm and} \quad \mathfrak{h}_{\lie} \not \cong \leib \left( \mathfrak{h} \right)\\
\K \quad {\rm if} \quad N \cong \leib \left( \mathfrak{h} \right)  \quad {\rm and} \quad \mathfrak{h}_{\lie} \not \cong \leib \left( \mathfrak{h} \right)\\
\K \oplus \K \quad {\rm if} \quad N \cong \mathfrak{h}_{\lie} \cong \leib \left( \mathfrak{h} \right)  \\
0 \quad {\rm otherwise}.
\end{cases}$$
\end{corol}

Denote by $H^{\star}\left( \mathfrak{g}, M \right)$ the Chevalley-Eilenberg cohomology of the Lie algebra ${\mathfrak g}$ with values in the left ${\mathfrak g}$-module $M$.

\begin{theo} \label{weyl}[Weyl's Theorem] (see Chapter VII in \cite{STM}).
Let $\mathfrak{g}$ be a simple Lie algebra and $M$ be a finite dimensional $\mathfrak{g}$-module. Then
$$H^{\star}\left( \mathfrak{g}, M \right) = H^{\star}\left( \mathfrak{g}, \K \right)\otimes M^{\mathfrak g} $$
\end{theo}

\begin{corol}\label{corol weyl}[Whitehead's Theorem]
Let $\mathfrak{g}$ be a simple Lie algebra and $N$ be a non-trivial simple finite dimensional $\mathfrak{g}$-module. Then
$$ H^{\star}\left( \mathfrak{g}, M \right) =0.$$
\end{corol}

\begin{lem} \label{whitehead}[Whitehead's Lemmas] (see Chapter VII, in \cite{STM})
Let $\mathfrak{g}$ be a semi-simple Lie algebra and $M$ be a finite dimensional $\mathfrak{g}$-module. Then
$$ H^{1}\left( \mathfrak{g}, M \right) = H^{2}\left( \mathfrak{g}, M \right)=0 .$$
\end{lem}

\subsection{Computation of the cohomology of a trivial Lie/Leibniz algebra}

In this subsection, we collect several results in order to compute the cohomology of the trivial Lie/Leibniz algebra.

\begin{defi}
Let $\mathfrak{g}$ be a Lie algebra. Let $U$ and $V$ be left $\mathfrak{g}$-modules. Then $\Hom \left( U, V \right)$ is a $\mathfrak{g}$-module via:
$$\left( x.f \right)\left( u \right)= -f\left( x.u \right)+ x.f\left(u \right)$$
$\forall x \in \mathfrak{g}$, $\forall u \in U$ and  $\forall f \in \Hom \left( U, V \right)$.  
\end{defi}

One of the important differences between Leibniz cohomology and Chevalley-Eilenberg cohomology of a Lie algebra ${\mathfrak g}$ is that the latter is bounded by the dimension of ${\mathfrak g}$, owing to the fact that it is defined on the exterior products of ${\mathfrak g}$:

\begin{lem}\label{dimension}
Let $\mathfrak{g}$ be a Lie algebra of dimension $n$ over a field $\K$. Then for any  $\mathfrak{g}$-module M,\\ $H^{p} \left(  \mathfrak{g}, M \right)= 0$ for $p \geq n+1$.
\end{lem}

\begin{lem} \label{Cheval}
Let $\mathfrak{h}=\K$ be the trivial Lie/Leibniz algebra and $M$ be a non-trivial simple $\mathfrak{h}$-bimodule.
In case the ${\mathfrak h}$-bimodule is trivial (but still simple), we denote it by $\K$.  Then,

\begin{center}
$$H^{p} \left(  \K, \K \right)= \begin{cases}
\K \quad {\rm if} \quad p=0,1 \\
0 \quad {\rm if} \quad p\geq 2.
\end{cases}$$ 
\end{center}

\begin{center}
$H^{p} \left(  \K, M \right)= 0$ \quad $\forall p\geq 0$ 
\end{center}

\end{lem}

\begin{proof}
This follows directly from Lemma \ref{dimension} and the direct computation of $H^0$ and $H^1$ for trivial and non-trivial $1$-dimensional modules.
\end{proof}

The cohomology for a trivial Lie/Leibniz algebra with values in any bimodule has been computed in Theorem 4.3 in \cite{JF-FW2}:

\begin{theo} \label{cohom}
Let $\mathfrak{h}=\K=\langle e\rangle$ be the $1$-dimensional Lie algebra, and let $M$ be a Leibniz $\mathfrak{h}$-bimodule. Then
$$ HL^{n}\left( \mathfrak{h}, M \right)\cong 
 \begin{cases}
M^{\mathfrak{h}} \quad {\rm if} \quad n=0 ,\\
M^{0}/ (M\cdot \mathfrak{h}) \quad {\rm if} \quad n \quad {\rm is} \quad {\rm odd} ,\\
M^{\mathfrak{h}}/ M_{0} \quad {\rm if} \quad n \quad {\rm is} \quad {\rm even} \quad {\rm and} \quad n \neq 0,
\end{cases}$$

(as $\K$-vector spaces) for every non-negative integer $n$, where
$$M^{0}=\lbrace m \in M  \mid  e.m + m.e=0 \rbrace .$$
Moreover, if $M$ is finite dimensional, then (as $\K$-vector spaces)
$$M^{0}/ (M\cdot \mathfrak{h}) \cong M^{\mathfrak{h}}/ M_{0}. $$
\end{theo}

From Theorem $\ref{cohom}$, one can deduce the cohomology with values in a simple bimodule $M$:

\begin{corol} \label{coho}
Let $\mathfrak{h}=\K$ be the trivial Lie/Leibniz algebra, and $M$ a simple left $\K$-module. Then

\begin{center}
$HL^{q}\left( \mathfrak{h}, \K \right)=\K \quad \forall q \geq 0 $ 

$HL^{q}\left( \mathfrak{h},  M^{s} \right)= 0 \quad \forall q \geq 0 $

$HL^{q}\left( \mathfrak{h},  M^{a} \right)= 
\begin{cases}
M^{a} \quad {\rm if} \quad q=0 \\
0 \quad {\rm otherwise}
\end{cases} $
\end{center}

\end{corol}

\begin{proof}
This is straight-forward by the computation of the spaces $M^{\mathfrak{h}}$, $M^{0}$, $(M\cdot \mathfrak{h})$ and $M_0$ in each case. 
\end{proof}

\subsection{Two useful spectral sequences}

Loday and Pirashvili (in \cite{JL-TP1}) have deviced a method for computing the Ext-groups between simple ${\mathfrak h}$-bimodules for a simple Lie algebra ${\mathfrak g}$. The method is based on Chevalley-Eilenberg- and Leibniz algebra cohomology computations and two change-of-rings spectral sequences whose $E_2$-terms have these cohomology spaces as ingredients. 

Later Mugni\'ery and Wagemann (in \cite{JM-FW2}) adapted the two spectral sequences to the Leibniz setting and computed with the same method the Ext-groups between simple bimodules for a simple Leibniz algebra ${\mathfrak h}$. We will always consider finite-dimensional Leibniz algebras ${\mathfrak h}$ and finite-dimensional ${\mathfrak h}$-bimodules $M$. It is well-known that in this setting, every bimodule admits a finite Jordan-H\"older series, which thus makes it in principle possible to compute the Ext-groups between arbitrary bimodules using long exact sequences knowing the Ext-groups between simple bimodules. Therefore we will focus on the latter.   
  
\begin{prop} \label{Spectral Seq}
(See Proposition 2.1, in \cite{JM-FW1})\\
Let $\mathfrak{h}$ be a Leibniz algebra, let $X$ be an $\mathfrak{h}$-bimodule, and $Y$ and $Z$ be left $\mathfrak{h}$-modules. There are two spectral sequences:

$$E_{2}^{pq}= H^{p} \left(  \mathfrak{h}_{\lie}, \Hom \left( Y, \HL^{q}\left( \mathfrak{h}, X \right)   \right) \right) \Rightarrow \Ext_{UL\left( \mathfrak{h} \right) }^{p+q}  \left( Y^{a}, X \right) ,$$
and
$$E_{2}^{pq}= H^{p} \left(  \mathfrak{h}_{\lie}, \Hom \left( Z, Ext_{UL\left( \mathfrak{h} \right) }^{q}\left( U{\mathfrak{h}_{\lie}}^{s}, X \right)   \right) \right) \Rightarrow \Ext_{UL\left( \mathfrak{h} \right) }^{p+q}  \left( Z^{s}, X \right). $$
\end{prop}

The two spectral sequences express the change-of-rings between the categories of left $U{\mathfrak h}_\lie$-modules and that of left $UL({\mathfrak h})$-modules. 
Note that Loday and Pirashvili showed in \cite{JL-TP2} (Theorem 3.4) that there is an isomorphism $$\Ext_{UL\left( \mathfrak{h} \right) }^{\ast}\left( U{\mathfrak{h}_{\lie}}^{a}, X \right)\cong \HL^{\ast}\left( \mathfrak{h}, X \right).$$

In order to use the second spectral sequence, one needs informations about $\Ext_{UL\left( \mathfrak{h} \right) }^{\ast}\left( U{\mathfrak{h}_{\lie}}^{s}, X \right) $. This is given by the following proposition (see \cite{JL-TP1} and \cite{JM-FW1}) giving a relation between Leibniz cohomology and the group $\Ext_{UL\left( \mathfrak{h} \right) }^{\ast}\left( U{\mathfrak{h}_{\lie}}^{s}, X \right)$: 

\begin{prop}
(see Proposition 2.2, in \cite{JM-FW1})\\
Let $\mathfrak{h}$ be a Leibniz algebra, let $M$ be an $\mathfrak{h}$-bimodule. Then there are isomorphisms:
$$\Ext_{UL\left( \mathfrak{h} \right) }^{q+1}\left( U{\mathfrak{h}_{\lie}}^{s}, M \right) \simeq  
\begin{cases}
\Hom \left(\mathfrak{h} , \HL^{q}\left( \mathfrak{h}, M \right)   \right) \quad {\rm for} \quad q > 0\\
\coker(f) \quad {\rm for} \quad q=0\\
\Ker(f) \quad {\rm for} \quad q=-1
\end{cases}$$

where $f: M \longrightarrow \Hom \left( \mathfrak{h}, \HL^{0}\left( {\mathfrak h}, M \right)   \right)$ is given by :\\
$$f\left(m  \right) \left( x \right)=  x\cdot m + m\cdot x$$ 
 $\forall x \in \mathfrak{h}$ and $\forall m \in M$. 
\end{prop}


\begin{corol} \label{prop22}

Let $\mathfrak{h}=\K$ be the trivial Lie/Leibniz algebra and $M$ be an $\mathfrak{h}$-bimodule. Then

$$\Ext_{UL\left( \K \right) }^{q}\left( U{\mathfrak{h}_{\lie}}^{s}, M \right) \simeq  
\begin{cases}
\Hom \left( \K, \HL^{q-1}\left( \K, M \right)   \right) \quad {\rm for} \quad q \geqslant 2\\
\coker(f) \quad {\rm for} \quad q=1\\
\Ker(f) \quad {\rm for} \quad q=0,
\end{cases}$$

where $f: M \longrightarrow \Hom \left( \K, \HL^{0}\left( \K, M \right)   \right)$ is given by:
$$f\left(m  \right) \left( e \right)=  e.m + m.e, $$
$\forall e \in \K$ and $\forall m \in M$. 
\end{corol}

\subsection{Computation of Ext-groups in the case $\K=\mathbb{C}$}

The following theorem is our first main result:


   \begin{theo} \label{comput}
 Let $\K=\mathbb{C}$ be the trivial Lie/Leibniz algebra. Let $M$ and $N$ be non-trivial simple one-dimensional $\K$-bimodules. Then:
 
 \begin{center}
 $\Ext_{UL\left( K \right) }^{n}\left(\K, \K \right)= 
\begin{cases}
\K \quad {\rm if} \quad n=0 \\
\K \oplus \K \quad {\rm if} \quad n \geq 1
\end{cases}$
 \end{center}
 
 \begin{center}
 $Ext_{UL\left( \K \right) }^{n}\left(M^{a}, M^{a} \right)=
 \K \quad {\rm for} \quad n=0,1$
 \end{center}
 
 \begin{center}
 $\Ext_{UL\left( \K \right) }^{n}\left( M^{s}, M^{s} \right)=
\K \quad {\rm for} \quad n=0,1$
 \end{center}
 
 All other groups $\Ext_{UL\left( \K \right) }^{n}\left( M, N \right)$ are zero.

 \end{theo}

\begin{proof}
We compute $\Ext_{UL\left(\K \right) }^{\ast}\left( M, N \right)$ for every combination of simple $1$-dimensional $\K$-bimodules and reduce it to the Chevalley-Eilenberg cohomology of the trivial Lie algebra. For this, we will study nine cases :

\textbf{$\bigstar$ Case 1: $M=N=\K$ the trivial $\K$-bimodule}

We apply Proposition \ref{Spectral Seq} by using the first spectral sequence :
$$E_{2}^{pq}= H^{p} \left(  \mathfrak{h}_{\lie}, \Hom \left( Y, \HL^{q}\left( \mathfrak{h}, X \right)   \right) \right) \Rightarrow \Ext_{UL\left( \mathfrak{h} \right) }^{p+q}  \left( Y^{a}, X \right). $$

We set $Y=X=\K$.
By Corollary $\ref{coho}$, one has $\HL^{q}\left( \K, \K \right)=\K $ \quad $\forall q \geq 0$. Then
$ E_{2}^{pq}= H^{p} \left(  \K, \Hom \left( \K, \HL^{q}\left( \K, \K \right)   \right) \right)= H^{p} \left(  \K, \Hom \left( \K, \K   \right) \right)=H^{p} \left(  \K, \K \right)$. 

By Lemma $\ref{Cheval}$, one has:
$$H^{p} \left(  \K, \K \right)= 
\begin{cases}
\K \quad {\rm if} \quad p=0,1 \\
0 \quad {\rm if} \quad p\geq 2.
\end{cases}$$

On the second page $E_2^{pq}$ of the spectral sequence, we have the second differential
$$d_{2}^{pq} : E_{2}^{pq} \longrightarrow E_{2}^{p+2, q-1}.$$
As there are only two non-zero columns on the second page, $d_2^{pq}$ must either start from zero or map to zero. This means that the spectral sequence collapses at the second page, and we obtain $E_{r}^{pq}=E_{r+1}^{pq}$ $\forall r \geqslant 2$, and thus
$$\Ext_{UL\left( \mathfrak{h} \right) }^{n}  \left( \K, \K \right)= \bigoplus_{p+q=n}E_{\infty}^{pq}=\bigoplus_{p+q=n}E_{2}^{pq}.$$

We are thus in the following situation:

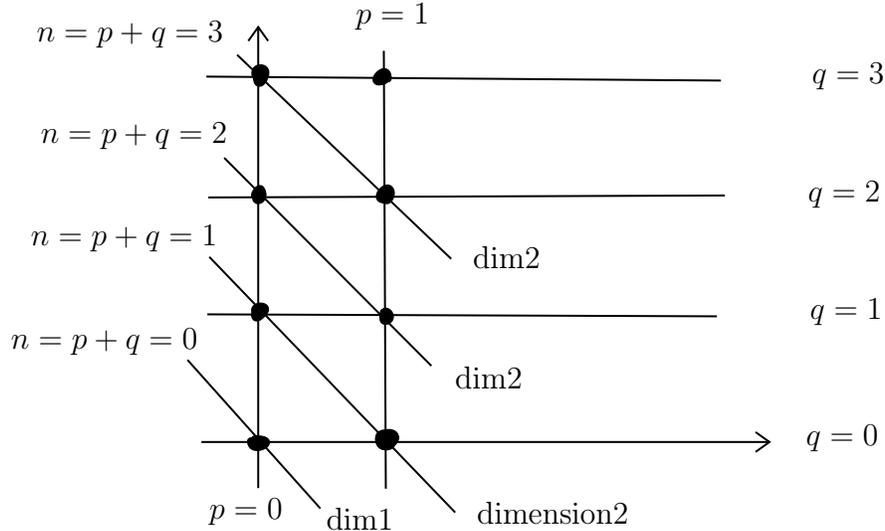
\begin{figure}[H]
	\centering
	\tikzset{every picture/.style={line width=0.75pt}} 

\begin{tikzpicture}[x=0.75pt,y=0.75pt,yscale=-1,xscale=1]

\draw  (207.5,244.43) -- (494,244.43)(236.15,35) -- (236.15,267.7) (487,239.43) -- (494,244.43) -- (487,249.43) (231.15,42) -- (236.15,35) -- (241.15,42)  ;
\draw    (299.5,47) -- (300.5,268) ;
\draw    (210,60) -- (469.5,62) ;
\draw    (210.5,180) -- (467.5,181) ;
\draw    (211,121) -- (471.5,120) ;
\draw    (200.5,203) -- (267.5,278) ;
\draw    (211.5,151) -- (335.5,280) ;
\draw    (219,101) -- (323.5,206) ;
\draw    (225.5,49) -- (333.5,152) ;
\draw  [line width=3] [line join = round][line cap = round] (236.5,60) .. controls (233.57,60) and (238.5,58.05) .. (238.5,61) ;
\draw  [line width=3] [line join = round][line cap = round] (300.5,120) .. controls (303.6,123.1) and (301,119) .. (299.5,119) ;
\draw  [line width=3] [line join = round][line cap = round] (301.5,245) .. controls (301.83,245) and (301.82,244.11) .. (301.5,244) .. controls (300.55,243.68) and (297.5,244) .. (298.5,244) .. controls (299.83,244) and (301.21,244.32) .. (302.5,244) .. controls (304.56,243.49) and (297.4,244) .. (301.5,244) ;
\draw  [line width=3] [line join = round][line cap = round] (299.5,118) .. controls (299.5,127.14) and (304.78,119.85) .. (302.5,117) .. controls (300.28,114.23) and (294.3,122) .. (299.5,122) ;
\draw  [line width=3] [line join = round][line cap = round] (237.5,61) .. controls (233.32,61) and (233.78,54.43) .. (238.5,56) .. controls (242.58,57.36) and (235.5,68.62) .. (235.5,60) ;
\draw  [line width=3] [line join = round][line cap = round] (301.5,242) .. controls (295.94,242) and (301.07,248.04) .. (303.5,245) .. controls (309.01,238.11) and (298.15,239.7) .. (297.5,241) .. controls (295.85,244.3) and (298.64,249.05) .. (303.5,245) .. controls (308.93,240.48) and (299.76,239.49) .. (298.5,242) .. controls (295.61,247.77) and (303.5,247.5) .. (303.5,245) ;
\draw  [line width=3] [line join = round][line cap = round] (301.5,61) .. controls (301.5,51.78) and (290.02,65.16) .. (299.5,62) .. controls (299.95,61.85) and (300.03,61) .. (300.5,61) ;
\draw  [line width=3] [line join = round][line cap = round] (299.5,179) .. controls (299.5,180.33) and (298.17,184.67) .. (301.5,183) .. controls (304.9,181.3) and (299.5,175.3) .. (299.5,181) ;
\draw  [line width=3] [line join = round][line cap = round] (234.5,180) .. controls (234.5,184.12) and (238.82,179.37) .. (239.5,178) .. controls (240.72,175.56) and (234.5,175.19) .. (234.5,178) ;
\draw  [line width=3] [line join = round][line cap = round] (235.5,120) .. controls (233.32,124.36) and (240.82,122.53) .. (237.5,117) .. controls (236.77,115.79) and (234.5,118.59) .. (234.5,120) ;
\draw  [line width=3] [line join = round][line cap = round] (232.5,244) .. controls (232.5,251.35) and (248.2,243) .. (234.5,243) ;

\draw (210,271.4) node [anchor=north west][inner sep=0.75pt]    {$p=0$};
\draw (284,21.4) node [anchor=north west][inner sep=0.75pt]    {$p=1$};
\draw (511,234.4) node [anchor=north west][inner sep=0.75pt]    {$q=0$};
\draw (513,170.4) node [anchor=north west][inner sep=0.75pt]    {$q=1$};
\draw (512,111.4) node [anchor=north west][inner sep=0.75pt]    {$q=2$};
\draw (514,51.4) node [anchor=north west][inner sep=0.75pt]    {$q=3$};
\draw (123,30.4) node [anchor=north west][inner sep=0.75pt]    {$n=p+q=3$};
\draw (125,81.4) node [anchor=north west][inner sep=0.75pt]    {$n=p+q=2$};
\draw (120,133.4) node [anchor=north west][inner sep=0.75pt]    {$n=p+q=1$};
\draw (110,185.4) node [anchor=north west][inner sep=0.75pt]    {$n=p+q=0$};
\draw (334,205) node [anchor=north west][inner sep=0.75pt]   [align=left] {dim2};
\draw (345,273) node [anchor=north west][inner sep=0.75pt]   [align=left] {dimension2};
\draw (343,144) node [anchor=north west][inner sep=0.75pt]   [align=left] {dim2};
\draw (269,276) node [anchor=north west][inner sep=0.75pt]   [align=left] {dim1};

\end{tikzpicture}
\end{figure}

Finally, $$\Ext_{UL\left( \K \right) }^{n}\left(\K, \K \right)= 
\begin{cases}
\K \quad {\rm if} \quad n=0, \\
\K \oplus \K \quad {\rm if} \quad n \geq 1.
\end{cases}$$

\textbf{$\bigstar$ Case 2: $M=\K$ the trivial $\K$-bimodule, and $N=N^{s}$ the non-trivial simple symmetric $\K$-bimodule}

We apply the first spectral sequence of Proposition \ref{Spectral Seq}:
$$E_{2}^{pq}= H^{p} \left(  \mathfrak{h}_{\lie}, \Hom \left( Y, \HL^{q}\left( \mathfrak{h}, X \right)   \right) \right) \Rightarrow \Ext_{UL\left( \mathfrak{h} \right) }^{p+q}  \left( Y^{a}, X \right). $$

We set $Y=\K$ and $X=N^{s}$. By Corollary \ref{coho}, $\HL^{q}\left( \K, N^{s}\right)=0 $ $\forall q \geq 0$.
Thus $E_{2}^{pq}= H^{p} \left( \K, \Hom \left( \K, \HL^{q}\left( \K, N^{s} \right)   \right) \right)= 0$.

Finally $$\Ext_{UL\left( \K \right) }^{n}\left(\K, N^{s} \right)= 0.$$

\textbf{$\bigstar$ Case 3: $M=\K$ the trivial $\K$-bimodule, and $N=N^{a}$ the non-trivial simple antisymmetric $\K$-bimodule}

We apply the first spectral sequence of Proposition \ref{Spectral Seq}:
$$E_{2}^{pq}= H^{p} \left(  \mathfrak{h}_{\lie}, \Hom \left( Y, \HL^{q}\left( \mathfrak{h}, X \right)   \right) \right) \Rightarrow \Ext_{UL\left( \mathfrak{h} \right) }^{p+q}  \left( Y^{a}, X \right). $$

We set $Y=\K$ and $X=N^{a}$. Since by Corollary \ref{coho} 
$$\HL^{q} \left(  \K, N^{a} \right)= \begin{cases}
N^{a} \quad {\rm if} \quad q=0 \\
0 \quad {\rm if} \quad q \geq 1,
\end{cases}$$
it follows that
$$E_{2}^{pq}= H^{p} \left(  \K, \Hom \left( \K, \HL^{q}\left( \K, N^{a} \right)   \right) \right)= 
\begin{cases}
H^{p} \left(  \K, \Hom \left( \K,  N^{a} \right) \right) \quad {\rm if} \quad q=0 \\
0 \quad {\rm if} \quad q \geq 1.
\end{cases}$$

Furthermore, $\Hom \left( \K,  N^{a} \right)\simeq N^{a}$ is a non-trivial $\K$-bimodule.
Using Lemma $\ref{Cheval}$, one has:
$$H^{p} \left(  \K, \Hom \left( \K,  N^{a} \right) \right)= H^{p} \left(  \K,  N^{a} \right)=0.$$
 
Finally 
$$\Ext_{UL\left( \K \right) }^{n}\left(\K, N^{a} \right)= 0.$$

\textbf{$\bigstar$ Case 4: $M= M^{a}$ the non-trivial simple antisymmetric $\K$-bimodule, and $N=\K$ the trivial $\K$-bimodule}
 
We apply the first spectral sequence of Proposition \ref{Spectral Seq}:
$$E_{2}^{pq}= H^{p} \left(  \mathfrak{h}_{\lie}, \Hom \left( Y, \HL^{q}\left( \mathfrak{h}, X \right)   \right) \right) \Rightarrow \Ext_{UL\left( \mathfrak{h} \right) }^{p+q}  \left( Y^{a}, X \right) .$$

We set $Y=M^{a}$ and $X=\K$. Since by Corollary \ref{coho}
$$\HL^{q} \left(  \K, \K \right)= \K$$
$\forall q \geq 0$, it follows that
$$E_{2}^{pq}= H^{p} \left(  \K, \Hom \left(M^{a}, \HL^{q}\left( \K, \K \right)   \right) \right)= H^{p} \left(  \K, \Hom \left( M^{a}, \K   \right) \right).$$

Moreover, in $H^{p} \left(  \K, \Hom \left( M^{a}, \K   \right) \right)$, $\K$ does not act trivially on $\Hom \left( M^{a}, \K   \right)$ and  $\Hom \left( M^{a}, \K   \right)\simeq M^{\star}$ is a non-trivial $\K$-bimodule.
Using Lemma \ref{Cheval}, one has:
$$H^{p} \left(  \K, \Hom \left( M^{a}, \K   \right) \right)=H^{p} \left(  \K, M^{\star} \right)=0.$$

Finally 
 $$\Ext_{UL\left( \K \right) }^{n}\left(M^{a}, \K \right)= 0.$$
 
\textbf{$\bigstar$ Case 5: $M= M^{a}$ the non-trivial simple antisymmetric $\K$-bimodule, $N= N^{s}$ the non-trivial simple symmetric $\K$-bimodule}
 
We apply the first spectral sequence of Proposition \ref{Spectral Seq}:
$$E_{2}^{pq}= H^{p} \left(  \mathfrak{h}_{\lie}, \Hom \left( Y, \HL^{q}\left( \mathfrak{h}, X \right)   \right) \right) \Rightarrow \Ext_{UL\left( \mathfrak{h} \right) }^{p+q}  \left( Y^{a}, X \right). $$

We set $Y=M^{a}$ and $X=N^{s}$. By Corollary \ref{coho} 
$$\HL^{q} \left(  \K, N^{s} \right)= 0$$ 
$\forall q \geq 0$. Then $$E_{2}^{pq}= H^{p} \left(  \K, \Hom \left(M^{a}, \HL^{q}\left( \K, N^{s} \right)   \right) \right)= 0.$$
 
Finally 
 $$\Ext_{UL\left( \K \right) }^{n}\left(M^{a}, N^{s} \right)= 0.$$

\textbf{$\bigstar$ Case 6: $M= M^{a}$ the non-trivial simple antisymmetric $\K$-bimodule, and $N= N^{a}$ the non-trivial simple antisymmetric $\K$-bimodule}

We apply the first spectral sequence of Proposition \ref{Spectral Seq}:
$$E_{2}^{pq}= H^{p} \left(  \mathfrak{h}_{\lie}, \Hom \left( Y, \HL^{q}\left( \mathfrak{h}, X \right)   \right) \right) \Rightarrow \Ext_{UL\left( \mathfrak{h} \right) }^{p+q}  \left( Y^{a}, X \right). $$

We set $Y=M^{a}$ and $X=N^{a}$. By Corollary \ref{coho} 
$$\HL^{q} \left(  \K, N^{a} \right)= \begin{cases}
N^{a} \quad {\rm if} \quad q=0 \\
0 \quad {\rm if} \quad q \geq 1.
\end{cases}$$
It follows that 
$$E_{2}^{pq}= H^{p} \left(  \K, \Hom \left(M^{a} , \HL^{q}\left( \K, N^{a} \right)   \right) \right)= 
\begin{cases}
H^{p} \left(  \K, \Hom \left( M^{a},  N^{a} \right) \right) \quad {\rm if} \quad q=0 \\
0 \quad {\rm if} \quad q \geq 1.
\end{cases}$$

In $H^{p} \left(  \K, \Hom \left( M^{a},  N^{a} \right) \right)$, one determines first if $\Hom \left( M^{a},  N^{a} \right)\simeq \K$ is a trivial or a non-trivial module. Let  $e \in \K$, $ 1_{m} \in M^{a}$, $ 1_{n} \in N^{a}$ and $ f \in \Hom \left( M^{a},  N^{a} \right)$. Let us denote the action in $M$ by setting $e.m=\lambda_{\rho}m$
and the action in $N$ by setting $e.n=\lambda_{\sigma}n$. 
Then, $\left( e.f \right)\left( m \right)= -f\left( e.m \right)+ e.f\left( m \right)= -f \lambda_{\rho}m + f \lambda_{\sigma}m = \left( \lambda_{\rho} - \lambda_{\sigma}\right)fm \neq 0$ if $\rho \neq \sigma $.

Conclusion:
$$ \begin{cases}
\Hom \left( M^{a}_{\rho},  N^{a}_{\sigma} \right)\simeq \K \quad {\rm is} \quad {\rm trivial} \quad {\rm if} \quad \rho = \sigma \\
\Hom \left( M^{a}_{\rho},  N^{a}_{\sigma} \right)\simeq \K \quad {\rm is} \quad {\rm not}\quad {\rm trivial} \quad {\rm if} \quad \rho \neq \sigma.
\end{cases}$$

Thus for $\rho \neq \sigma$, we have
$$E_{2}^{pq}= H^{p} \left(  \K, \K   \right)= 0,$$
and for $\rho = \sigma$, we have
$$E_{2}^{pq}= H^{p} \left(  \K, \K   \right)=
\begin{cases}
\K \quad {\rm for} \quad p=0,1 \quad {\rm and} \quad q=0 \\
0 \quad {\rm otherwise}.
\end{cases}$$

Finally if $\rho = \sigma$, 
$$\Ext_{UL\left( \K \right) }^{n}\left(M^{a}, M^{a} \right)=
 \begin{cases}
\K \quad {\rm for} \quad n=0,1  \\
0 \quad {\rm otherwise}.
\end{cases}$$

\textbf{$\bigstar$ Case 7: $M= M^{s}$ the non-trivial simple symmetric $\K$-bimodule, and $N=\K$ the trivial simple $\K$-bimodule}

From now on, we apply the second spectral sequence of Proposition \ref{Spectral Seq}:
$$E_{2}^{pq}= H^{p} \left(  \mathfrak{h}_{\lie}, \Hom \left( Z, \Ext_{UL\left( {\mathfrak h} \right) }^{q}\left( U{\mathfrak{h}_{\lie}}^{s}, X \right)   \right) \right) \Rightarrow \Ext_{UL\left( \mathfrak{h} \right) }^{p+q}  \left( Z^{s}, X \right) .$$

We set $Z=M^{s}$ and $X=\K$ and thus
$$E_{2}^{pq}= H^{p} \left(\K, \Hom \left( M^{s}, \Ext_{UL\left( \K \right) }^{q}\left( U{\K}^{s}, \K \right)   \right) \right).$$

By Corollary \ref{prop22}, one has :
$$\Ext_{UL\left( \K \right) }^{q}\left( U{\K}^{s}, \K \right) \simeq  
\begin{cases}
\Hom \left( \K, \HL^{q-1}\left( \K, \K \right)   \right) \quad {\rm for} \quad q \geqslant 2\\
\coker(f) \quad {\rm for} \quad q=1\\
\Ker(f) \quad {\rm for} \quad q=0.
\end{cases}$$

where $f: \K \longrightarrow \Hom \left( \K, \HL^{0}\left( \K, \K \right)   \right)\simeq \Hom \left( \K, \K \right)\simeq \K$ is given by :
$$f\left(m  \right) \left( e \right)= e\cdot m + m\cdot e$$
$\forall e \in K$ and $\forall m \in K$.
In this case, $f$ is the zero map, because $\K$ is a trivial bimodule. So, $\Ker(f)=\K$  and $\coker(f)=\K$.
Therefore, 
$$\Ext_{UL\left( \K \right) }^{q}\left( U{\K}^{s}, \K \right)=\K$$
$\forall q \geqslant 0$.  And then, 
$$E_{2}^{pq}= H^{p} \left( \K, \Hom \left( M^{s}, \Ext_{UL\left( \K \right) }^{q}\left( U{\K}^{s}, \K \right)   \right) \right)= H^{p} \left( \K, \Hom \left( M^{s}, \K \right) \right).$$

Moreover, in $H^{p} \left(  \K, \Hom \left( M^{s}, \K   \right) \right)$, $\K$ does not act trivially on $\Hom \left( M^{s}, \K   \right)$ and  $\Hom \left( M^{s}, \K   \right)\simeq M^{\star}$ is a non-trivial $\K$-module.
Using Lemma \ref{Cheval}, one has:
$$H^{p} \left(  \K, \Hom \left( M^{s}, \K   \right) \right)=H^{p} \left(  \K, M^{\star} \right)=0.$$

Finally
 $$\Ext_{UL\left( \K \right) }^{n}\left(M^{s}, \K \right)= 0.$$

\textbf{$\bigstar$ Case 8: $M= M^{s}$ the non-trivial simple symmetric $\K$-bimodule, and $N=N^{s}$ the non-trivial simple symmetric $\K$-bimodule}

We apply the second spectral sequence of Proposition \ref{Spectral Seq}:
$$ E_{2}^{pq}= H^{p} \left( {\mathfrak h}_{\lie}, \Hom \left( Z, \Ext_{UL\left( {\mathfrak h} \right) }^{q}\left( U{{\mathfrak h}_{\lie}}^{s}, X \right)   \right) \right) \Rightarrow \Ext_{UL\left( {\mathfrak h} \right) }^{p+q}  \left( Z^{s}, X \right). $$

We set $Z=M^{s}$ and $X=N^{s}$. This gives
$$ E_{2}^{pq}= H^{p} \left( \K, \Hom \left( M^{s}, \Ext_{UL\left( \K \right) }^{q}\left( U{\K}^{s}, N^{s}\right)   \right) \right).$$

Applying Corollary \ref{prop22}, one has:
$$Ext_{UL\left( \K \right) }^{q}\left( U{\K}^{s}, N^{s} \right) \simeq  
\begin{cases}
\Hom \left( \K, \HL^{q-1}\left( \K,  N^{s} \right)   \right)= 0 \quad {\rm for} \quad q \geqslant 2\\
\coker(f) = 0 \quad {\rm for} \quad q=1\\
\Ker(f) =  N^{s} \quad {\rm for} \quad q=0.
\end{cases}$$

where $f:  N^{s} \longrightarrow \Hom \left( \K, \HL^{0}\left( \K,  N^{s} \right) \right)= \lbrace 0\rbrace $   is the zero map. Therefore, 
$$ E_{2}^{pq}= H^{p} \left( \K, \Hom \left( M^{s}, \Ext_{UL\left( \K \right) }^{q}\left( U{\K}^{s}, N^{s}\right)   \right) \right)=
\begin{cases}
H^{p} \left( \K, \Hom \left( M^{s}, N^{s} \right)\right) \quad {\rm for} \quad q=0\\
 0 \quad {\rm otherwise}.
\end{cases}$$

As for the case 6,
$$ \begin{cases}
\Hom \left( M^{s},  N^{s} \right)\simeq \K \quad {\rm is} \quad {\rm trivial} \quad {\rm if} \quad \rho = \sigma \\
\Hom \left( M^{s},  N^{s} \right)\simeq \K \quad {\rm is} \quad {\rm not}\quad {\rm trivial} \quad {\rm if} \quad \rho \neq \sigma.
\end{cases}$$
Thus for $\rho \neq \sigma$,
$$E_{2}^{pq}= H^{p} \left(  \K, \K   \right)= 0,$$
and for $\rho = \sigma$,
$$E_{2}^{pq}= H^{p} \left(  \K, \K   \right)=
\begin{cases}
\K \quad {\rm for} \quad p=0,1 \quad {\rm and} \quad q=0 \\
0 \quad {\rm otherwise}.
\end{cases}$$

Finally if $\rho = \sigma$, 
 $$\Ext_{UL\left( \K \right) }^{n}\left(M^{s}, M^{s} \right)=
 \begin{cases}
\K \quad {\rm for} \quad n=0,1  \\
0 \quad {\rm otherwise}.
\end{cases}$$

\textbf{$\bigstar$ Case 9: $M= M^{s}$ the non-trivial simple symmetric $\K$-bimodule, $N=N^{a}$ the non-trivial simple antisymmetric $\K$-bimodule}

We apply the second spectral sequence of Proposition \ref{Spectral Seq}:
$$E_{2}^{pq}= H^{p} \left(  \mathfrak{h}_{\lie}, \Hom \left( Z, \Ext_{UL\left( \mathfrak{h} \right) }^{q}\left( U{\mathfrak{h}_{\lie}}^{s}, X \right)   \right) \right) \Rightarrow \Ext_{UL\left( \mathfrak{h} \right) }^{p+q}  \left( Z^{s}, X \right) .$$

We set $Z=M^{s}$ and $X=N^{a}$ and this gives
$$ E_{2}^{pq}= H^{p} \left( \K, \Hom \left( M^{s}, \Ext_{UL\left( \K \right) }^{q}\left( U{\K}^{s}, N^{a}\right)   \right) \right).$$
By Corollary \ref{coho}, 
$$\HL^{q-1} \left(  \K, N^{a} \right)= \begin{cases}
N^{a} \quad {\rm if} \quad q=1 \\
0 \quad {\rm otherwise}.
\end{cases}$$
Applying Corollary \ref{prop22}, it follows that
$$\Ext_{UL\left( \K \right) }^{q}\left( U{\K}^{s}, N^{a} \right) \simeq  
\begin{cases}
\Hom \left( \K, \HL^{q-1}\left( \K,  N^{a} \right)   \right)= 0 \quad {\rm for} \quad q \geqslant 2\\
\coker(f) = 0 \quad {\rm for} \quad q=1\\
\Ker(f) =  0 \quad {\rm for} \quad q=0.
\end{cases}$$

where $f: N^{a} \longrightarrow \Hom \left( \K, \HL^{0}\left( \K, N^{a} \right)   \right)\simeq \Hom \left( \K, N^{a} \right)\simeq N^{a}$ is given by:
$$f\left(m  \right) \left( e \right)= e\cdot m + m\cdot e = e\cdot m$$
$\forall e \in \K$ and $\forall m \in N^{a}$. 

So, $E_{2}^{pq}=0$ for all $p,q$. 

Finally, 
 $$\Ext_{UL\left( \K \right) }^{n}\left(M^{s}, N^{a} \right)=
 0.$$ 
 \end{proof}
 
 \begin{rem}
\begin{itemize}
\item[(a)] Note that the theorem implies that the Ext-dimension of the category of bimodules (finite-dimensional or not) is infinite. This contrasts with the category of finite-dimensional bimodules over a finite-dimensional complex simple Leibniz algebra where the Ext dimension is $2$. 
\item[(b)] It would be interesting to compute these Ext-groups (and then the Gabriel quivers) also for different fields, like ${\mathbb R}$, ${\mathbb Q}$ or $\F_q$. In these cases, there are more simple objects. 
\end{itemize} 
 \end{rem}

 \subsection{The Gabriel quiver of $L\left( \mathbb{K} \right)$}

Let $\mathfrak{h}$ be a Leibniz algebra. Denote by $L\left( \mathfrak{h} \right)$ the category of finite-dimensional $\mathfrak{h}$-bimodules.

 \begin{defi} \label{quiv}[Gabriel quiver] (see Section 4, in \cite{JL-TP1})
A {\it quiver} is a directed graph. Let \textbf{A} be a $\K$-linear abelian category whose objects have finite length. Let us denote by \textbf{Q(A)} its {\it Gabriel quiver}:
The vertices of the Gabriel quiver \textbf{Q(A)} are the isomorphism classes of simple objects of \textbf{A}.
If $S_{1}$ and $S_{2}$ are two simple objects, then \textbf{Q(A)} has exactly $\dim_{\K}\Ext_{A }^{1}\left( S_{1}, S_{2} \right)$ arrows from (the class of) $S_{1}$ to (the class of) $S_{2}$.
 \end{defi}
 
 As discussed before, one knows that the simple objects of $L\left(\mathbb{K} \right)$ are $1$-dimensional for the trivial Leibniz algebra $\K=\C$, and that they are given by modules of the kind $M^s$ or $M^a$. Let $M$ be a simple object of $L\left(\mathbb{K} \right)$. Therefore, there are three possibilities for $M$:
 \begin{itemize}
 \item[(1)] $M= \K$ the trivial bimodule,
\item[(2)] $M= M^{a}$ the non-trivial antisymmetric bimodule,
\item[(3)] $M= M^{s}$ the non-trivial symmetric bimodule.
\end{itemize}
Thus the vertices of $Q\left(L\left(\mathbb{K} \right)\right)$ are $\K$, $M^{a}$ and $M^{s}$.
Furthermore, by Theorem $\ref{comput}$, it follows that:
 \begin{itemize}
 \item[(1)] $\dim \Ext_{UL\left( \K \right) }^{1}\left(\K, \K \right)= 2$,
\item[(2)] $\dim \Ext_{UL\left( \K \right) }^{1}\left(M^{a}, M^{a} \right)= 1$,
\item[(3)] $\dim \Ext_{UL\left( \K \right) }^{1}\left(M^{s}, M^{s} \right)= 1$.
\end{itemize}
Therefore, $Q\left(L\left(\mathbb{K} \right)\right)$ looks as follows :

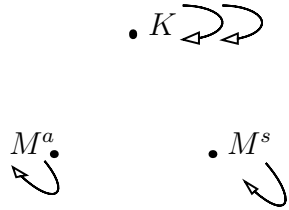
\begin{figure}[H]
	\centering
	\tikzset{every picture/.style={line width=0.75pt}} 

\begin{tikzpicture}[x=0.75pt,y=0.75pt,yscale=-1,xscale=1]

\draw  [fill={rgb, 255:red, 255; green, 255; blue, 255 }  ,fill opacity=1 ] (244.96,165.06) .. controls (249.36,164.19) and (247.27,156.75) .. (240.3,148.44) -- (240.3,148.44) .. controls (247.27,156.75) and (249.36,164.19) .. (244.96,165.06) -- cycle ;\draw  [fill={rgb, 255:red, 255; green, 255; blue, 255 }  ,fill opacity=1 ] (244.96,165.06) .. controls (241.41,165.77) and (234.73,161.99) .. (228.62,156.14) -- (231.21,155.63) -- (224.37,151.61) -- (226.03,156.66) -- (228.62,156.14) .. controls (234.73,161.99) and (241.41,165.77) .. (244.96,165.06) -- cycle ;
\draw  [fill={rgb, 255:red, 255; green, 255; blue, 255 }  ,fill opacity=1 ] (360.93,170.36) .. controls (363.84,167.78) and (359.8,158.44) .. (351.9,149.49) -- (351.9,149.49) .. controls (359.8,158.44) and (363.84,167.78) .. (360.93,170.36) -- cycle ;\draw  [fill={rgb, 255:red, 255; green, 255; blue, 255 }  ,fill opacity=1 ] (360.93,170.36) .. controls (358.61,172.4) and (352.61,169.56) .. (346.28,163.82) -- (347.63,162.63) -- (341.35,158.81) -- (344.93,165.01) -- (346.28,163.82) .. controls (352.61,169.56) and (358.61,172.4) .. (360.93,170.36) -- cycle ;
\draw  [fill={rgb, 255:red, 255; green, 255; blue, 255 }  ,fill opacity=1 ] (350,78.7) .. controls (350,73.9) and (341.05,70) .. (330,70) -- (330,70) .. controls (341.05,70) and (350,73.9) .. (350,78.7) -- cycle ;\draw  [fill={rgb, 255:red, 255; green, 255; blue, 255 }  ,fill opacity=1 ] (350,78.7) .. controls (350,82.3) and (344.97,85.39) .. (337.8,86.71) -- (337.8,84.11) -- (330,87.4) -- (337.8,89.31) -- (337.8,86.71) .. controls (344.97,85.39) and (350,82.3) .. (350,78.7) -- cycle ;
\draw  [fill={rgb, 255:red, 255; green, 255; blue, 255 }  ,fill opacity=1 ] (330,78.6) .. controls (330,73.85) and (321.05,70) .. (310,70) -- (310,70) .. controls (321.05,70) and (330,73.85) .. (330,78.6) -- cycle ;\draw  [fill={rgb, 255:red, 255; green, 255; blue, 255 }  ,fill opacity=1 ] (330,78.6) .. controls (330,82.09) and (325.16,85.1) .. (318.2,86.45) -- (318.2,83.65) -- (310,87.2) -- (318.2,89.25) -- (318.2,86.45) .. controls (325.16,85.1) and (330,82.09) .. (330,78.6) -- cycle ;
\draw  [line width=3] [line join = round][line cap = round] (285,84) .. controls (285,84.33) and (285,84.67) .. (285,85) ;
\draw  [line width=3] [line join = round][line cap = round] (245,145) .. controls (245,145) and (245,145) .. (245,145) ;
\draw  [line width=3] [line join = round][line cap = round] (325,145) .. controls (325,145) and (325,145) .. (325,145) ;

\draw (291,72.4) node [anchor=north west][inner sep=0.75pt]    {$K$};
\draw (221,132.4) node [anchor=north west][inner sep=0.75pt]    {$M^{a}$};
\draw (331,132.4) node [anchor=north west][inner sep=0.75pt]    {$M^{s}$};

\end{tikzpicture}
	\caption{Gabriel quiver when $\mathfrak{h}= \K = \mathbb{C}  $}
\end{figure}


\section{Gabriel quiver for $\mathfrak{h}=V_{n} \times_{hs} \mathfrak{sl}_{2}(\C)$}

\subsection{Preliminaries}

\begin{prop}\label{prop hemi}[Feldvoss] (see Theorem 7.12, in \cite{JF})
Every simple Leibniz algebra $\mathfrak{h}$ over the field $\mathbb{C}$ is a hemi-semidirect product of a simple Lie algebra $\mathfrak{g}$ and a simple $\mathfrak{g}$-bimodule $\leib \left( \mathfrak{h} \right)$, i.e.  $\mathfrak{h}=\mathfrak{g} \times_{hs} \leib \left( \mathfrak{h} \right)$.
\end{prop}

\textbf{Example:}
The Lie algebra $\mathfrak{g}=\mathfrak{sl}_{2}\left( \mathbb{C} \right)$ is simple. 
Moreover, for any $m\geq0$, there exists a unique simple  $\mathfrak{g}$-module $V_{m}$ whose dimension is $m+1$. Applying Proposition $\ref{prop hemi}$, the hemi-semidirect product $\mathfrak{h}=V_{m} \times_{hs} \mathfrak{sl}_{2}(\C)$ is a simple Leibniz algebra such that $\mathfrak{h}_{\lie}=\mathfrak{sl}_{2}\left( \mathbb{C}\right)$ and $\leib \left( \mathfrak{h} \right)= V_{m}$.

\begin{rem} \label{dual irred} (On the dual of a simple module)
\begin{itemize}
\item[(i)] Let $V_{n}$ be a simple $\mathfrak{sl}_{2}(\C)$-module. Then
$V_{n}^{\star} \cong V_{n}$. 
Indeed in dimension $n+1$, there exists only one simple $\mathfrak{sl}_{2}(\C)$-module, namely $V_{n}$.
\item[(ii)] The (co)adjoint module is $\mathfrak{sl}_{2}(\C)^{\star} \cong \mathfrak{sl}_{2}(\C) \cong V_{2} $.
\end{itemize}
\end{rem}

\subsection{Computation of Ext-groups} 

In this part, we collect several results in order to compute Ext-groups.

\begin{theo}\label{JM-FW}[Mugni\'ery-Wagemann] (see Theorem 2.3 in \cite{JM-FW2})
Let $\mathfrak{h}$ be a finite dimensional simple Leibniz algebra over $\C$ such that $\leib \left( \mathfrak{h} \right) \not \cong \mathfrak{h}_{\lie}$.

All groups  $\Ext_{UL\left( \mathfrak{h} \right) }^{2}\left(M, N \right)$ between simple finite dimensional $\mathfrak{h}$-bimodules are zero, except 
$$\Ext_{UL\left( \mathfrak{h} \right) }^{2}\left(M^{s}, N^{a} \right),$$
with $M \in \lbrace \leib \left( \mathfrak{h} \right)^{\star}, \mathfrak{h}_{\lie}^{\star} \rbrace$ and $N \in \lbrace \leib \left( \mathfrak{h} \right), \mathfrak{h}_{\lie} \rbrace$, which is $1$-dimensional.
Moreover, we have:

$\bullet$ $\Ext_{UL\left( \mathfrak{h} \right) }^{1}\left(M^{s}, \K \right)$, and $\Ext_{UL\left( \mathfrak{h} \right) }^{1}\left(\K, N^{a} \right)$ are $1$-dimensional for $M$ and $N \in \lbrace \leib \left( \mathfrak{h} \right), {\mathfrak h}_{\lie} \rbrace$;

$\bullet$ $\Ext_{UL\left( \mathfrak{h} \right) }^{1}\left(M^{s}, N^{a} \right)\simeq \Hom_{U  \mathfrak{h}_{\lie}}\left(M, \widehat{N} \right)$, where
$$ \widehat{N}= \coker \left( h: N \longrightarrow \Hom\left(\mathfrak{h}, N \right) \right) \quad h\left( n \right)\left( x \right)= x\cdot n$$

$\bullet$ All other groups  $\Ext_{UL\left( \mathfrak{h} \right) }^{1}\left(M, N \right)$ between simple finite dimensional $\mathfrak{h}$-bimodules $M$ and $N$ are zero.
\end{theo} 

\begin{rem}
Note that the condition $\leib \left( \mathfrak{h} \right) \not \cong \mathfrak{h}_{\lie}$ does not appear in \cite{JM-FW2}, but it is clear from the proof that this is implicitly imposed.  Before continuing, we will provide a complement to Theorem \ref{JM-FW}, namely we will compute the groups $\Ext_{UL(\mathfrak{h})}^{n}(M,N)$ without the restriction $\leib \left( \mathfrak{h} \right) \not \cong \mathfrak{h}_{\lie}$.
\end{rem}

\begin{theo} \label{complement} (Complement to Theorem \ref{JM-FW})
Let $\mathfrak{h}$ be a finite dimensional simple Leibniz algebra over $\C$. Then:

All groups  $\Ext_{UL\left( \mathfrak{h} \right) }^{2}\left(M, N \right)$ between simple finite dimensional $\mathfrak{h}$-bimodules are zero, except 
$$\Ext_{UL\left( \mathfrak{h} \right) }^{2}\left(M^{s}, N^{a} \right),$$
with $M \in \lbrace \leib \left( \mathfrak{h} \right)^{\star}, \mathfrak{h}_{\lie}^{\star} \rbrace$ and $N \in \lbrace \leib \left( \mathfrak{h} \right), \mathfrak{h}_{\lie} \rbrace$, which is $1$-dimensional if ${\mathfrak h}_\lie\not\cong \leib({\mathfrak h})$. The Ext-group is $4$-dimensional if ${\mathfrak h}_\lie\cong \leib({\mathfrak h})\cong N$ and 
${\mathfrak h}_\lie^\star\cong \leib({\mathfrak h})^\star\cong M$. 

$\bullet$ $\Ext_{UL\left( \mathfrak{h} \right) }^{1}\left(M^{s}, \K \right)$, and $\Ext_{UL\left( \mathfrak{h} \right) }^{1}\left(\K, N^{a} \right)$ are $1$- or $2$-dimensional for $M \in \lbrace \leib \left( \mathfrak{h} \right)^{\star}, \mathfrak{h}_{\lie}^{\star} \rbrace$ and $N \in \lbrace \leib \left( \mathfrak{h} \right), \mathfrak{h}_{\lie} \rbrace$

$\bullet$ $\Ext_{UL\left( \mathfrak{h} \right) }^{1}\left(M^{s}, N^{a} \right)\simeq \Hom_{U\mathfrak{h}_{\lie}}\left(M, \widehat{N} \right)$, where
$$ \widehat{N}= \coker \left( h: N \longrightarrow \Hom\left(\mathfrak{h}, N \right) \right) \quad h\left( n \right)\left( x \right)= x\cdot n.$$

$\bullet$ All other groups  $\Ext_{UL\left( \mathfrak{h} \right) }^{1}\left(M, N \right)$ between simple finite dimensional $\mathfrak{h}$-bimodules $M$ and $N$ are zero.
\end{theo}

\begin{proof}
The only cases where the hypothesis $\leib \left( \mathfrak{h} \right) \not \cong \mathfrak{h}_{\lie}$ comes into consideration in the proof of Theorem \ref{JM-FW} were case 3, case 6 and case 8.
 
$\blacktriangleright$ $\Ext_{UL\left( \mathfrak{h} \right) }^{n}\left(\K, N^{a} \right)$ (see case 3 in the proof of Theorem 2.3, in \cite{JM-FW2})

In this case, $\K$ is trivial and $N$ is non-trivial antisymmetric.
We use the first spectral sequence:
$$ E_{2}^{pq}= H^{p} \left(  \mathfrak{h}_{\lie}, \Hom \left( Y, \HL^{q}\left( \mathfrak{h}, X \right)   \right) \right) \Rightarrow \Ext_{UL\left( \mathfrak{h} \right) }^{p+q}  \left( Y^{a}, X \right) $$
We set $Y=\K$ and $X=N^{a}$. Then we have 
$$E_{2}^{pq}= H^{p} \left(  \mathfrak{h}_{\lie}, \Hom \left( \K, \HL^{q}\left( \mathfrak{h}, N^{a} \right)   \right) \right)$$
Furthermore
$$\HL^{q}\left( \mathfrak{h}, N^{a} \right)=
 \begin{cases}
N \quad {\rm if} \quad q=0  \\
\Hom_{U{\mathfrak{h}_{\lie}}} \left( \mathfrak{h}, N \right) \quad {\rm if} \quad q=1\\
0 \quad {\rm if} \quad  q>1 .
\end{cases}$$

Thus $$E_{2}^{pq}= H^{p} \left(  \mathfrak{h}_{\lie}, \Hom \left( \K, \HL^{q}\left( \mathfrak{h}, N^{a} \right)   \right) \right)=
 \begin{cases}
H^{p} \left(  \mathfrak{h}_{\lie}, \Hom_{U{\mathfrak{h}_{\lie}}} \left( \mathfrak{h}, N \right) \right)     \quad {\rm if} \quad q=1\\
0 \quad {\rm otherwise}.
\end{cases}$$
Therefore the spectral sequence collapses and we obtain 
$$\Ext_{UL\left( \mathfrak{h} \right) }^{n}\left(\K, N^{a} \right)= H^{n-1} \left(  \mathfrak{h}_{\lie}, \Hom_{U{\mathfrak{h}_{\lie}}} \left( \mathfrak{h}, N \right)  \right).$$
Using Weyl's theorem (Theorem \ref{weyl}), we deduce
$$\Ext_{UL\left( \mathfrak{h} \right) }^{n}\left(\K, N^{a} \right)= H^{n-1} \left(  \mathfrak{h}_{\lie}, \K \right) \otimes \left( \Hom_{U{\mathfrak{h}_{\lie}}} \left( \mathfrak{h}, N \right)  \right)^{\mathfrak{h}_{\lie}},$$
which implies in turn
$$\Ext_{UL\left( \mathfrak{h} \right) }^{n}\left(\K, N^{a} \right)= H^{n-1} \left(  \mathfrak{h}_{\lie}, \K \right) \otimes \Hom_{{\mathfrak{h}_{\lie}}} \left( \mathfrak{h}, N \right),$$
because $\Hom_{U{\mathfrak{h}_{\lie}}} ( \mathfrak{h}, N )^{\mathfrak{h}_\lie} = \Hom_{U{\mathfrak{h}_{\lie}}} ( \mathfrak{h}, N )$ by definition of the action on the Hom space. 

Applying Corollary \ref{corol schur} to $\Hom_{U{\mathfrak{h}_{\lie}}} \left( \mathfrak{h}, N \right)$, we obtain:

$$\Ext_{UL\left( \mathfrak{h} \right) }^{n}\left(\K, N^{a} \right)=
\begin{cases}
H^{n-1} \left(  \mathfrak{h}_{\lie}, \K \right)\quad {\rm if} \quad \left( N \cong \mathfrak{h}_{\lie} \quad {\rm or} \quad N \cong \leib \left( \mathfrak{h} \right) \right) \quad {\rm and} \quad \leib \left( \mathfrak{h} \right) \not \cong \mathfrak{h}_{\lie} \\
H^{n-1} \left(  \mathfrak{h}_{\lie}, \K \right) \oplus H^{n-1} \left(  \mathfrak{h}_{\lie}, \K \right)\quad {\rm if} \quad N \cong \mathfrak{h}_{\lie} \cong \leib \left( \mathfrak{h} \right)\\
0 \quad {\rm otherwise}.
\end{cases}$$

$\blacktriangleright$ $\Ext_{UL\left( \mathfrak{h} \right) }^{n}\left(M^{s}, \K \right)$
(see case 6 in the proof of Theorem 2.3 in \cite{JM-FW2})

In this case, $\K$ is trivial and $M$ is non-trivial symmetric.
We use the second spectral sequence:
$$E_{2}^{pq}= H^{p} \left(  \mathfrak{h}_{\lie}, \Hom \left( M, \Ext_{UL\left( \mathfrak{h} \right) }^{q}\left( U{\mathfrak{h}_{\lie}}^{s}, K \right)   \right) \right) \Rightarrow \Ext_{UL\left( \mathfrak{h} \right) }^{p+q}  \left( M^{s}, \K \right).$$
We set $Z=M$ and $X=\K$ and obtain for the codomain of the coefficient term
$$\Ext_{UL\left( \mathfrak{h} \right) }^{q}\left( U{\mathfrak{h}_{\lie}}^{s}, \K \right) \simeq  
\begin{cases}
\Hom \left( \mathfrak{h}, \HL^{q-1}\left( \mathfrak{h}, \K \right)   \right) = 0 \quad {\rm for} \quad q > 1\\
\coker(f) = \mathfrak{h}^{\star} \quad {\rm for} \quad q=1\\
\Ker(f) = \K \quad {\rm for} \quad q=0. 
\end{cases}$$
where $f: \K \longrightarrow \Hom \left( \mathfrak{h}, \HL^{0}\left( \mathfrak{h}, \K \right)   \right)= \mathfrak{h}^{\star}$ is given by:
$$f\left(m  \right) \left( x \right)=  x\cdot m +   m\cdot x $$ 
$\forall x \in \mathfrak{h}$ and $\forall m \in \K$. 
We deduce for the $E_2$-term:
$$E_{2}^{pq}= 
\begin{cases}
H^{p} \left(  \mathfrak{h}_{\lie}, \Hom \left( M, \mathfrak{h}^{\star} \right)   \right) \quad {\rm for} \quad q=1\\
0 \quad {\rm otherwise}.
\end{cases}$$
By Weyl's theorem (Theorem \ref{weyl}), therefore
$$\Ext_{UL\left( \mathfrak{h} \right) }^{n}\left( M^{s}, \K \right)= H^{n-1} \left(  \mathfrak{h}_{\lie}, \K \right) \otimes  \Hom_{U{\mathfrak{h}_{\lie}}} \left( M, \mathfrak{h}^{\star} \right).$$ 

Applying Corollary \ref{corol schur} to $\Hom_{U{\mathfrak{h}_{\lie}}} \left( M, \mathfrak{h}^{\star} \right)$, one obtains:
$$\Ext_{UL\left( \mathfrak{h} \right) }^{n}\left( M^{s}, \K \right)=
\begin{cases}
H^{n-1} \left(  \mathfrak{h}_{\lie}, \K \right)\quad {\rm if} \quad \left(  M \cong \mathfrak{h}_{\lie}^{\star} \quad {\rm or} \quad M \cong \leib \left( \mathfrak{h} \right)^{\star}\right) \quad {\rm and} \quad \mathfrak{h}_{\lie}^{\star} \not \cong \leib \left( \mathfrak{h} \right)^{\star}\\
H^{n-1} \left(  \mathfrak{h}_{\lie}, \K \right) \oplus H^{n-1} \left(  \mathfrak{h}_{\lie}, \K \right)\quad  {\rm if} \quad M \cong \mathfrak{h}_{\lie}^{\star} \cong \leib \left( \mathfrak{h} \right)^{\star}\\
0 \quad {\rm otherwise}.
\end{cases}$$

$\blacktriangleright$ $\Ext_{UL\left(\mathfrak{h} \right) }^{1}\left(M^{s}, N^{a} \right)\simeq \Hom_{U \left( \mathfrak{h}_{\lie}\right)}\left(M, \widehat{N} \right)$

In this case, $M$ is non-trivial symmetric and $N$ is non-trivial antisymmetric.
Here the proof given in Theorem 2.3, in \cite{JM-FW2} for $\Ext_{UL\left(\mathfrak{h} \right) }^{1}\left(M^{s}, N^{a} \right)$ still goes through for $\Ext^1$.

$\blacktriangleright$ Concerning the $\Ext^2$, we have for the $E_2$-term of the spectral sequence as in case 8 in the proof given in Theorem 2.3, in \cite{JM-FW2}
$$E^{p2}_2=H^p({\mathfrak h}_\lie,\K)\otimes\Hom_{U{\mathfrak h}_\lie}(M,\Hom({\mathfrak h},\Hom_{U{\mathfrak h}_\lie}({\mathfrak h},N))).$$

Now, we have 
$$\Hom_{U{\mathfrak h}_\lie}({\mathfrak h},N)= 
\begin{cases}
\K \quad {\rm if} \quad N\cong{\mathfrak h}_\lie \quad {\rm and } \quad {\mathfrak h}_\lie\not\cong \leib({\mathfrak h}) \\
\K \quad {\rm if} \quad N\cong\leib({\mathfrak h}) \quad {\rm and } \quad {\mathfrak h}_\lie\not\cong \leib({\mathfrak h}) \\
\K\oplus\K \quad {\rm if} \quad  N\cong {\mathfrak h}_\lie\cong \leib({\mathfrak h}) \\
0 \quad {\rm otherwise}.
\end{cases}$$
Thus the space $\Hom_{U{\mathfrak h}_\lie}(M,\Hom({\mathfrak h},\Hom_{U{\mathfrak h}_\lie}({\mathfrak h},N)))$ becomes $\Hom_{U{\mathfrak h}_\lie}(M,{\mathfrak h}^\star)$, $\Hom_{U{\mathfrak h}_\lie}(M,{\mathfrak h}^\star)$, $\Hom_{U{\mathfrak h}_\lie}(M,{\mathfrak h}^\star\oplus{\mathfrak h}^\star )$ or zero according to the four cases. If $M$ is isomorphic to ${\mathfrak h}_\lie^\star$ or $\leib({\mathfrak h})^\star$ in one of the first two cases, this gives an Ext of dimension $1$. But in the third case, this gives an Ext of dimension four if $M\cong{\mathfrak h}_\lie^\star\cong\leib({\mathfrak h})^\star$ !
\end{proof}

\subsection{Computation of Ext-groups for $\mathfrak{h}=V_{n} \times_{hs} \mathfrak{sl}_{2}(\C)$}
Let us now specialize to the Lie algebra $\mathfrak{sl}_{2}(\C)$.
Denote
$$ \mathfrak{h}=V_{n}\times_{hs} \mathfrak{sl}_{2}(\C),\quad M=V_{p},\quad{\rm and}\quad N=V_{m}.$$
where $V_{n}$, $V_{p}$, $V_{m}$ are the simple $\mathfrak{sl}_{2}$-modules of dimension $n+1$, $p+1$ and $m+1$ respectively, and $\mathfrak{h}=V_{n}\times_{hs} \mathfrak{sl}_{2}(\C)$ is the hemi-semidirect product Leibniz algebra. We will need the following lemma:

\begin{lem}\label{clebs}[Clebsch-Gordon] 
Let $V_{m}$ and $V_{n}$ be the simple $\mathfrak{sl}_{2}(\C)$-modules of dimension $m+1$ and $n+1$. Then
$$V_{m}\otimes V_{n}= V_{m+n} \oplus V_{m+n-2} \oplus\ldots\oplus V_{\mid m-n \mid}.$$
\end{lem}

Let us compute  $\Ext_{UL\left( \mathfrak{h} \right) }^{1}\left(M^{s}, N^{a} \right)$ for ${\mathfrak h}=V_{n}\times_{hs} \mathfrak{sl}_{2}(\C)$, $M=V_p$ and $N=V_m$ by applying Theorem \ref{complement}. We obtain
$$Ext_{UL\left( \mathfrak{h} \right) }^{1}\left(M^{s}, N^{a} \right)\simeq \Hom_{U \mathfrak{sl}_{2}(\C)}\left(M, \widehat{N} \right),$$ 
where $M$ and $N$ are simple $\mathfrak{sl}_{2}(\C)$-modules and
\begin{align*}
\widehat{N} &=& \coker \left( f: N \longrightarrow \Hom\left(V_{n} \oplus \mathfrak{sl}_{2}(\C), N \right) \right) \quad f\left( n \right)\left( x \right)=  x\cdot n \\
&=&\coker \left( f: V_{m} \longrightarrow \Hom\left(V_{n} \oplus V_{2}, V_{m} \right) \right) \quad f\left( n \right)\left( x \right)=  x\cdot n.
\end{align*}

In order to compute $\widehat{N}= \coker \left( f: V_{m} \longrightarrow \Hom\left(V_{n} \oplus V_{2}, V_{m} \right) \right)$, let us first compute ${\rm Im}(f)$ and $\Hom\left(V_{n} \oplus V_{2}, V_{m} \right)$:

$\bullet$ \textbf{Computation of ${\rm Im}(f)$:}

\begin{itemize}
\item[(1)] $m\neq 0$, i.e. $N=V_{m}$ is a non-trivial simple $\mathfrak{sl}_{2}(\C)$-module. Then the image ${\rm Im}(f)$ must be a non-trivial submodule of $\Hom\left(V_{n} \oplus V_{2}, V_{m} \right) $ isomorphic to $V_m$.  
\item[(2)] $m = 0$, i.e. $N=V_{0}$ is trivial and the image is zero. 
\end{itemize}

$\bullet$ \textbf{Computation of $\Hom\left(V_{n} \oplus V_{2}, V_{m} \right)$:}

$$\Hom\left(V_{n} \oplus V_{2}, V_{m} \right)\cong \left(V_{n}^{\star} \oplus V_{2}^{\star}\right) \otimes V_{m} \cong \left(V_{n} \oplus V_{2}\right) \otimes V_{m}= \left(V_{n}\otimes V_{m} \right)\oplus \left(V_{2}\otimes V_{m} \right).$$ 

By Lemma $\ref{clebs}$ (Clebsch-Gordon), one has:
\begin{itemize}
\item $V_{n}\otimes V_{m}= V_{m+n} \oplus V_{m+n-2} \oplus\ldots\oplus V_{\mid m-n \mid}$\\
\item $V_{2}\otimes V_{m}= V_{m+2} \oplus V_{m} \oplus V_{ m-2}$ \quad if $m\geq2$\\
\item $V_{2}\otimes V_{m}= V_{3} \oplus V_{1}$ \quad if $m=1$\\
\item $V_{2}\otimes V_{m}= V_{2} $ \quad if $m=0$\\
\end{itemize}
Thus
$$\Hom\left(V_{n} \oplus V_{2}, V_{m} \right)=
\begin{cases}
V_{m+n} \oplus V_{m+n-2} \oplus\ldots\oplus V_{\mid m-n \mid} \oplus V_{m+2} \oplus V_{m} \oplus V_{ m-2 } \quad {\rm if} \quad m\geq2\\
 V_{n+1} \oplus V_{n-1} \oplus V_{3} \oplus V_{1} \quad {\rm if} \quad m=1\\
 V_{n} \oplus V_{2} \quad {\rm if} \quad m=0.
\end{cases}$$

$\bullet$ \textbf{Computation of $\widehat{N}= \coker(f)$ }

Since $\widehat{N}= \coker(f) = \Hom\left(V_{n} \oplus V_{2}, V_{m} \right)/ {\rm Im}(f)$, we have

$$\widehat{N}=
\begin{cases}
V_{m+n} \oplus V_{m+n-2} \oplus\ldots\oplus V_{\mid m-n \mid} \oplus V_{m+2} \oplus V_{m-2} \quad {\rm if} \quad m\geq2\\
 V_{n+1} \oplus V_{n-1} \oplus V_{3} \quad {\rm if} \quad m=1\\
 V_{n} \oplus V_{2} \quad {\rm if} \quad m=0.
\end{cases}$$

$\bullet$ \textbf{Computation of $\dim\Ext_{UL\left( \mathfrak{h} \right) }^{1}\left(M^{s}, N^{a} \right)$}

 \begin{theo} \label{quiver hemi}
Let $\mathfrak{h}=V_{n}\times_{hs}\mathfrak{sl}_{2}(\C)$ be the simple Leibniz algebra. Let $M=V_{p}$ and $N=V_{m}$ be the simple $\mathfrak{h}$-bimodules. Then

\begin{itemize}
\item[$\blacktriangleright$] For $m\geqslant2$
$$\dim\Ext_{UL\left( \mathfrak{h} \right) }^{1}\left(V_{p}^{s}, V_{m}^{a} \right)=
\begin{cases}
2 \quad {\rm if} \quad p \in \lbrace m+n, m+n-2,\ldots, \mid m-n \mid  \rbrace \quad {\rm and} \quad p \in \lbrace m+2, m-2  \rbrace\\
1 \quad {\rm if} \quad p \in \lbrace m+n, m+n-2,\ldots, \mid m-n \mid  \rbrace \quad {\rm and} \quad p \notin \lbrace m+2, m-2  \rbrace\\
1 \quad {\rm if} \quad p \in \lbrace m+2, m-2  \rbrace \quad {\rm and} \quad  p \notin \lbrace m+n, m+n-2,\ldots, \mid m-n \mid  \rbrace \\
0 \quad {\rm otherwise}.
\end{cases}$$
\item[$\blacktriangleright$] For $m = 1$
$$\dim\Ext_{UL\left( \mathfrak{h} \right) }^{1}\left(V_{p}^{s}, V_{m}^{a} \right)=
\begin{cases}
2 \quad {\rm if} \quad p \in \lbrace n+1, n-1, 3 \rbrace \quad {\rm and} \quad \lbrace p \rbrace \cap \lbrace 3 \rbrace \cap \lbrace n+1, n-1  \rbrace \neq \emptyset\\
1 \quad {\rm if} \quad p \in \lbrace n+1, n-1, 3 \rbrace \quad {\rm and} \quad \lbrace p \rbrace \cap \lbrace 3 \rbrace \cap \lbrace n+1, n-1  \rbrace = \emptyset\\
0 \quad {\rm if} \quad p \notin \lbrace n+1, n-1, 3 \rbrace .
\end{cases}$$
\item[$\blacktriangleright$] For $m = 0$
$$\dim\Ext_{UL\left( \mathfrak{h} \right) }^{1}\left(V_{p}^{s}, V_{0} \right)=
\begin{cases}
2 \quad {\rm if} \quad p =n=2 \\
1 \quad {\rm if} \quad p \in \lbrace n, 2 \rbrace \quad {\rm and} \quad n\neq2 \\
0 \quad {\rm if} \quad p \notin \lbrace n, 2 \rbrace .
\end{cases}$$
\end{itemize}
\end{theo}

\begin{proof}
For $m\geqslant2$
$$\dim\Ext_{UL\left( \mathfrak{h} \right) }^{1}\left(M^{s}, N^{a} \right)=\dim\Ext_{UL\left( \mathfrak{h} \right) }^{1}\left(V_{p}^{s}, V_{m}^{a} \right)=\dim\Hom_{U{\mathfrak{sl}_{2}}(\C) } \left(V_{p}, \widehat{V_{m}} \right)$$
and
$$\dim\Ext_{UL\left( \mathfrak{h} \right) }^{1}\left(V_{p}^{s}, V_{m}^{a} \right)=\dim\Hom_{U{\mathfrak{sl}_{2}}(\C) } \left(V_{p}, V_{m+n} \oplus V_{m+n-2} \oplus\ldots\oplus V_{\mid m-n \mid} \oplus V_{m+2} \oplus V_{m-2} \right).$$

Applying Schur's lemma (Lemma $\ref{schur}$) to $\Hom_{U{\mathfrak{sl}_{2}}(\C) } \left(V_{p}, V_{m+n} \oplus V_{m+n-2} \oplus\ldots\oplus V_{\mid m-n \mid} \oplus V_{m+2} \oplus V_{m-2} \right)$, the result follows.

For $m=1$
$$\dim\Ext_{UL\left( \mathfrak{h} \right) }^{1}\left(M^{s}, N^{a} \right)=\dim\Ext_{UL\left( \mathfrak{h} \right) }^{1}\left(V_{p}^{s}, V_{m}^{a} \right)=\dim\Hom_{U{\mathfrak{sl}_{2}} (\C)} \left(V_{p}, \widehat{V_{m}} \right)$$
and
$$\dim\Ext_{UL\left( \mathfrak{h} \right) }^{1}\left(V_{p}^{s}, V_{m}^{a} \right)=\dim\Hom_{U{\mathfrak{sl}_{2}}(\C) } \left(V_{p}, V_{n+1} \oplus V_{n-1} \oplus V_{3} \right).$$

Applying Schur's lemma (Lemma $\ref{schur}$) to $\Hom_{U{\mathfrak{sl}_{2}} (\C)} \left(V_{p}, V_{n+1} \oplus V_{n-1} \oplus V_{3} \right)$, the result follows.

For $m=0$
$$\dim\Ext_{UL\left( \mathfrak{h} \right) }^{1}\left(V_{p}^{s}, V_{0} \right)=\dim\Hom_{U{\mathfrak{sl}_{2}}(\C) } \left(V_{p}, \widehat{V_{0}} \right)$$
and
$$\dim\Ext_{UL\left( \mathfrak{h} \right) }^{1}\left(V_{p}^{s}, V_{0} \right)=\dim\Hom_{U{\mathfrak{sl}_{2}}(\C) } \left(V_{p}, V_{n} \oplus V_{2} \right).$$

Applying Schur's lemma (Lemma $\ref{schur}$) to $\Hom_{U{\mathfrak{sl}_{2}}(\C) } \left(V_{p}, V_{n} \oplus V_{2} \right)$, the result follows. 
\end{proof}

\subsection{Construction of the quiver for $ \mathfrak{h}=V_{n}\times_{hs} \mathfrak{sl}_{2}(\C)$}

In this subsection, we will compute the Gabriel quiver of $L\left( \mathfrak{h} \right)$, with $\mathfrak{h}=V_{n}\times_{hs} \mathfrak{sl}_{2}(\C)$ for $n = 1$ and $n = 2$.

\textbf{Case1: $n=1$} and thus $\mathfrak{h}=V_{1}\times_{hs} \mathfrak{sl}_{2}(\C)$.
In this case, we have
$$\widehat{N}=
\begin{cases}
V_{m+1} \oplus V_{m-1} \oplus V_{m+2} \oplus V_{m-2} \quad {\rm if} \quad m\geq2\\
 V_{2} \oplus V_{0} \oplus V_{3} \quad {\rm if} \quad m=1\\
 V_{1} \oplus V_{2} \quad {\rm if} \quad m=0.
\end{cases}$$
Applying Theorem $\ref{quiver hemi}$, one has
$$\dim\Ext_{UL\left( \mathfrak{h} \right) }^{1}\left(V_{p}^{s}, V_{m}^{a} \right)=
\begin{cases}
1 \quad {\rm if} \quad p \in \lbrace m+1, m-1, m+2, m-2 \rbrace \quad {\rm and} \quad m \geq 2\\
1 \quad {\rm if} \quad p \in \lbrace 0, 2, 3 \rbrace \quad {\rm and} \quad m=1\\
1 \quad {\rm if} \quad p \in \lbrace 1, 2 \rbrace \quad {\rm and} \quad m=0\\
0 \quad {\rm otherwise}.
\end{cases}$$
Concerning the number of arrows, this means that 
$V_{0}$ receives two arrows: one from $V_{1}^{s}$ and one from $V_{2}^{s}$,
$V_{1}^{a}$ receives three arrows: one from $V_{0}^{s}$, one from $V_{2}^{s}$ and one from $V_{3}^{s}$, and
for $m\geq2$, $V_{m}^{a}$ receives four arrows: one from $V_{m+1}^{s}$, one from $V_{m-1}^{s}$, one from $V_{m+2}^{s}$ and one from $V_{m-2}^{s}$.

The quiver for $L\left( \mathfrak{h}=V_{1}\times_{hs} \mathfrak{sl}_{2}(\C)\right)$ looks as follows:

\begin{figure}[H]
	\centering
	\input{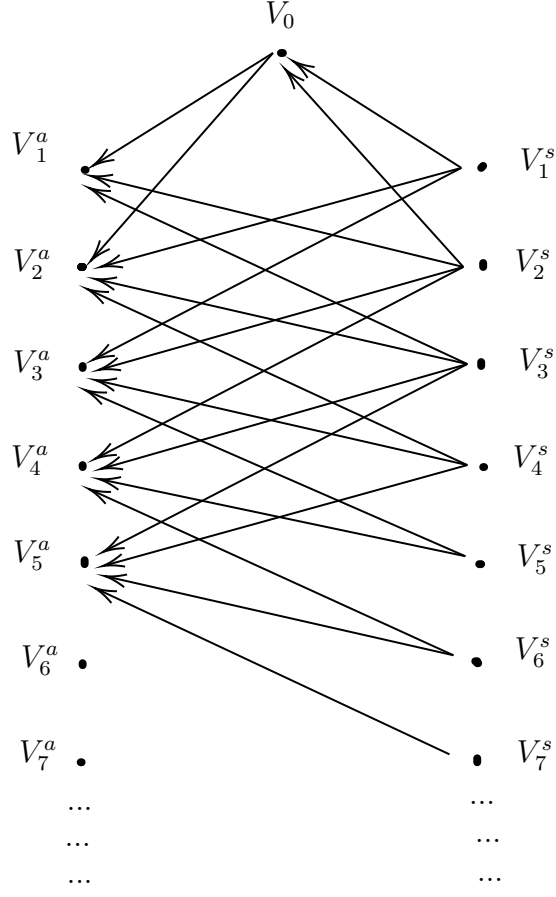}
	\caption{Gabriel quiver when $\mathfrak{h}= V_{1}\times_{hs} \mathfrak{sl}_{2}(\C)$}
\end{figure}

\textbf{Case 2: $n=2$} and thus $\mathfrak{h}= V_{2}\times_{hs} \mathfrak{sl}_{2}(\C)$.
In this case, we have
$$\widehat{N}=
\begin{cases}
V_{m+2} \oplus V_{m} \oplus V_{m-2} \oplus V_{m+2} \oplus V_{m-2} \quad {\rm if} \quad m\geq2\\
 V_{3} \oplus V_{1} \oplus V_{3} \quad {\rm if} \quad m=1\\
 V_{2} \oplus V_{2} \quad{\rm if} \quad m=0.
\end{cases}$$
Applying Theorem $\ref{quiver hemi}$, one has
$$\dim\Ext_{UL\left( \mathfrak{h} \right) }^{1}\left(V_{p}^{s}, V_{m}^{a} \right)=
\begin{cases}
2 \quad {\rm if} \quad p \in \lbrace m+2, m-2 \rbrace \quad {\rm and} \quad m \geq 2\\
2 \quad {\rm if} \quad p=3 \quad {\rm and} \quad m =1\\
2 \quad {\rm if} \quad p=2 \quad {\rm and} \quad m =0\\
1 \quad {\rm if} \quad p=m \quad {\rm and} \quad m \geq 2\\
1 \quad {\rm if} \quad p=1 \quad {\rm and} \quad m =1\\
0 \quad {\rm otherwise}.
\end{cases}$$
Concerning the number of arrows, this means that 
$V_{0}$ receives two arrows from $V_{2}^{s}$,
$V_{1}^{a}$ receives two arrows from $V_{3}^{s}$, one arrow from $V_{1}^{s}$, and 
for $m\geq2$, $V_{m}^{a}$ receives two arrows from $V_{m+2}^{s}$, two arrows from $V_{m-2}^{s}$ but one arrow from $V_{m}^{s}$.

The quiver for $\mathfrak{h}= V_{2}\times_{hs} \mathfrak{sl}_{2}(\C)$ looks as follows:

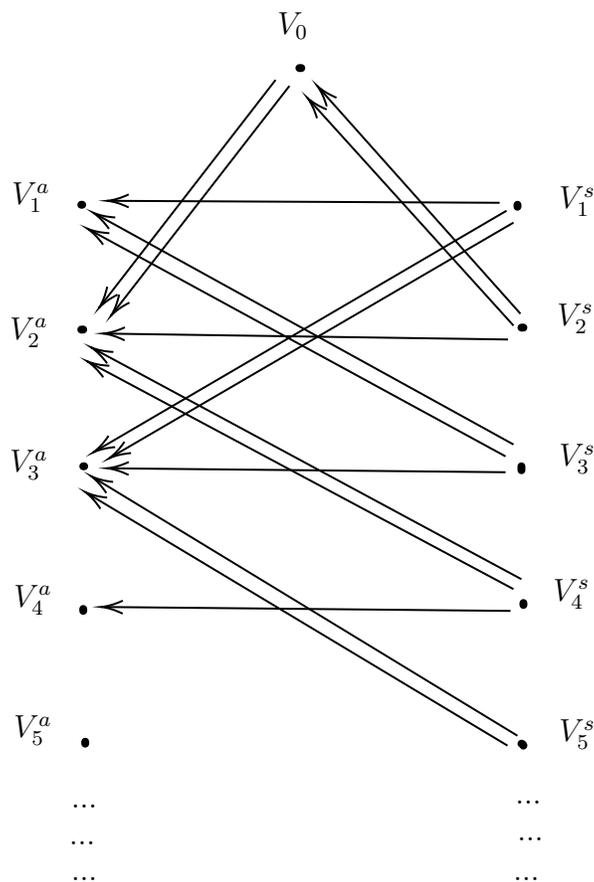
\begin{figure}[H]
	\centering
	\tikzset{every picture/.style={line width=0.75pt}} 

\begin{tikzpicture}[x=0.75pt,y=0.75pt,yscale=-1,xscale=1]

\draw    (416,169) -- (314.83,55.49) ;
\draw [shift={(313.5,54)}, rotate = 48.29] [color={rgb, 255:red, 0; green, 0; blue, 0 }  ][line width=0.75]    (10.93,-3.29) .. controls (6.95,-1.4) and (3.31,-0.3) .. (0,0) .. controls (3.31,0.3) and (6.95,1.4) .. (10.93,3.29)   ;
\draw    (412.5,175) -- (309.84,61.48) ;
\draw [shift={(308.5,60)}, rotate = 47.88] [color={rgb, 255:red, 0; green, 0; blue, 0 }  ][line width=0.75]    (10.93,-3.29) .. controls (6.95,-1.4) and (3.31,-0.3) .. (0,0) .. controls (3.31,0.3) and (6.95,1.4) .. (10.93,3.29)   ;
\draw    (412.5,235) -- (202.25,118.97) ;
\draw [shift={(200.5,118)}, rotate = 28.89] [color={rgb, 255:red, 0; green, 0; blue, 0 }  ][line width=0.75]    (10.93,-3.29) .. controls (6.95,-1.4) and (3.31,-0.3) .. (0,0) .. controls (3.31,0.3) and (6.95,1.4) .. (10.93,3.29)   ;
\draw    (408.5,241) -- (199.26,126.96) ;
\draw [shift={(197.5,126)}, rotate = 28.59] [color={rgb, 255:red, 0; green, 0; blue, 0 }  ][line width=0.75]    (10.93,-3.29) .. controls (6.95,-1.4) and (3.31,-0.3) .. (0,0) .. controls (3.31,0.3) and (6.95,1.4) .. (10.93,3.29)   ;
\draw    (409.75,182) -- (207.5,179.03) ;
\draw [shift={(205.5,179)}, rotate = 0.84] [color={rgb, 255:red, 0; green, 0; blue, 0 }  ][line width=0.75]    (10.93,-3.29) .. controls (6.95,-1.4) and (3.31,-0.3) .. (0,0) .. controls (3.31,0.3) and (6.95,1.4) .. (10.93,3.29)   ;
\draw    (408.5,249) -- (209.5,247.02) ;
\draw [shift={(207.5,247)}, rotate = 0.57] [color={rgb, 255:red, 0; green, 0; blue, 0 }  ][line width=0.75]    (10.93,-3.29) .. controls (6.95,-1.4) and (3.31,-0.3) .. (0,0) .. controls (3.31,0.3) and (6.95,1.4) .. (10.93,3.29)   ;
\draw    (416.5,303) -- (201.26,186.95) ;
\draw [shift={(199.5,186)}, rotate = 28.33] [color={rgb, 255:red, 0; green, 0; blue, 0 }  ][line width=0.75]    (10.93,-3.29) .. controls (6.95,-1.4) and (3.31,-0.3) .. (0,0) .. controls (3.31,0.3) and (6.95,1.4) .. (10.93,3.29)   ;
\draw    (412.5,308) -- (199.26,193.94) ;
\draw [shift={(197.5,193)}, rotate = 28.14] [color={rgb, 255:red, 0; green, 0; blue, 0 }  ][line width=0.75]    (10.93,-3.29) .. controls (6.95,-1.4) and (3.31,-0.3) .. (0,0) .. controls (3.31,0.3) and (6.95,1.4) .. (10.93,3.29)   ;
\draw    (414.5,382) -- (201.21,252.04) ;
\draw [shift={(199.5,251)}, rotate = 31.35] [color={rgb, 255:red, 0; green, 0; blue, 0 }  ][line width=0.75]    (10.93,-3.29) .. controls (6.95,-1.4) and (3.31,-0.3) .. (0,0) .. controls (3.31,0.3) and (6.95,1.4) .. (10.93,3.29)   ;
\draw    (409.5,387) -- (198.21,260.03) ;
\draw [shift={(196.5,259)}, rotate = 31] [color={rgb, 255:red, 0; green, 0; blue, 0 }  ][line width=0.75]    (10.93,-3.29) .. controls (6.95,-1.4) and (3.31,-0.3) .. (0,0) .. controls (3.31,0.3) and (6.95,1.4) .. (10.93,3.29)   ;
\draw    (411,319) -- (206.5,317.02) ;
\draw [shift={(204.5,317)}, rotate = 0.55] [color={rgb, 255:red, 0; green, 0; blue, 0 }  ][line width=0.75]    (10.93,-3.29) .. controls (6.95,-1.4) and (3.31,-0.3) .. (0,0) .. controls (3.31,0.3) and (6.95,1.4) .. (10.93,3.29)   ;
\draw    (292.5,51) -- (204.72,164.42) ;
\draw [shift={(203.5,166)}, rotate = 307.74] [color={rgb, 255:red, 0; green, 0; blue, 0 }  ][line width=0.75]    (10.93,-3.29) .. controls (6.95,-1.4) and (3.31,-0.3) .. (0,0) .. controls (3.31,0.3) and (6.95,1.4) .. (10.93,3.29)   ;
\draw    (299.5,54) -- (210.73,168.42) ;
\draw [shift={(209.5,170)}, rotate = 307.81] [color={rgb, 255:red, 0; green, 0; blue, 0 }  ][line width=0.75]    (10.93,-3.29) .. controls (6.95,-1.4) and (3.31,-0.3) .. (0,0) .. controls (3.31,0.3) and (6.95,1.4) .. (10.93,3.29)   ;
\draw  [line width=3] [line join = round][line cap = round] (194.5,114) .. controls (194.5,114) and (194.5,114) .. (194.5,114) ;
\draw  [line width=3] [line join = round][line cap = round] (414.5,114) .. controls (414.5,114.33) and (414.5,114.67) .. (414.5,115) ;
\draw  [line width=3] [line join = round][line cap = round] (195.5,177) .. controls (195.17,177) and (194.83,177) .. (194.5,177) ;
\draw  [line width=3] [line join = round][line cap = round] (305.5,45) .. controls (305.17,45) and (304.83,45) .. (304.5,45) ;
\draw  [line width=3] [line join = round][line cap = round] (417.5,176) .. controls (417.17,176) and (416.83,176) .. (416.5,176) ;
\draw  [line width=3] [line join = round][line cap = round] (195.5,246) .. controls (195.5,246) and (195.5,246) .. (195.5,246) ;
\draw  [line width=3] [line join = round][line cap = round] (416.5,248) .. controls (416.5,247.33) and (416.5,246.67) .. (416.5,246) ;
\draw  [line width=3] [line join = round][line cap = round] (196.5,385) .. controls (196.5,385.33) and (196.5,385.67) .. (196.5,386) ;
\draw  [line width=3] [line join = round][line cap = round] (416.5,386) .. controls (416.97,386) and (417.5,386.53) .. (417.5,387) ;
\draw  [line width=3] [line join = round][line cap = round] (195.5,319) .. controls (195.5,318.67) and (195.5,318.33) .. (195.5,318) ;
\draw  [line width=3] [line join = round][line cap = round] (417.5,316) .. controls (417.5,315.67) and (417.5,315.33) .. (417.5,315) ;
\draw    (409,117) -- (202.23,237.99) ;
\draw [shift={(200.5,239)}, rotate = 329.67] [color={rgb, 255:red, 0; green, 0; blue, 0 }  ][line width=0.75]    (10.93,-3.29) .. controls (6.95,-1.4) and (3.31,-0.3) .. (0,0) .. controls (3.31,0.3) and (6.95,1.4) .. (10.93,3.29)   ;
\draw    (412.5,123) -- (210.23,240.99) ;
\draw [shift={(208.5,242)}, rotate = 329.74] [color={rgb, 255:red, 0; green, 0; blue, 0 }  ][line width=0.75]    (10.93,-3.29) .. controls (6.95,-1.4) and (3.31,-0.3) .. (0,0) .. controls (3.31,0.3) and (6.95,1.4) .. (10.93,3.29)   ;
\draw    (405.5,113) -- (209.5,112.01) ;
\draw [shift={(207.5,112)}, rotate = 0.29] [color={rgb, 255:red, 0; green, 0; blue, 0 }  ][line width=0.75]    (10.93,-3.29) .. controls (6.95,-1.4) and (3.31,-0.3) .. (0,0) .. controls (3.31,0.3) and (6.95,1.4) .. (10.93,3.29)   ;

\draw (292,15.4) node [anchor=north west][inner sep=0.75pt]    {$V_{0}$};
\draw (434,103.4) node [anchor=north west][inner sep=0.75pt]    {$V_{1}^{s}$};
\draw (158,101.4) node [anchor=north west][inner sep=0.75pt]    {$V_{1}^{a}$};
\draw (433,162.4) node [anchor=north west][inner sep=0.75pt]    {$V_{2}^{s}$};
\draw (157,169.4) node [anchor=north west][inner sep=0.75pt]    {$V_{2}^{a}$};
\draw (434,233.4) node [anchor=north west][inner sep=0.75pt]    {$V_{3}^{s}$};
\draw (157,236.4) node [anchor=north west][inner sep=0.75pt]    {$V_{3}^{a}$};
\draw (159,304.4) node [anchor=north west][inner sep=0.75pt]    {$V_{4}^{a}$};
\draw (431,301.4) node [anchor=north west][inner sep=0.75pt]    {$V_{4}^{s}$};
\draw (434,372.4) node [anchor=north west][inner sep=0.75pt]    {$V_{5}^{s}$};
\draw (159,370.4) node [anchor=north west][inner sep=0.75pt]    {$V_{5}^{a}$};
\draw (187,434) node [anchor=north west][inner sep=0.75pt]   [align=left] {...};
\draw (188,415) node [anchor=north west][inner sep=0.75pt]   [align=left] {...};
\draw (412,413) node [anchor=north west][inner sep=0.75pt]   [align=left] {...};
\draw (413,432) node [anchor=north west][inner sep=0.75pt]   [align=left] {...};
\draw (188,452) node [anchor=north west][inner sep=0.75pt]   [align=left] {...};
\draw (411,452) node [anchor=north west][inner sep=0.75pt]   [align=left] {...};

\end{tikzpicture}
	\caption{Gabriel quiver when $\mathfrak{h}= V_{2}\times_{hs} \mathfrak{sl}_{2}(\C)$}
\end{figure}


\end{document}